\documentclass[11pt,a4paper,english]{article}
\usepackage{amsmath, latexsym, amsfonts, amssymb, amsthm, amscd}
\usepackage[hmargin=2.5cm,vmargin=2.5cm]{geometry}
\usepackage{pdfsync}
\usepackage{graphicx}
\usepackage{epsfig}
\usepackage{subfigure}
\usepackage{float}
\usepackage{authblk}
\newtheorem{theorem}{Theorem}
\newtheorem{corollary}{Corollary}
\newtheorem{proposition}{Proposition}
\newtheorem{lemma}{Lemma}

\newtheorem{definition}{Definition}

\newcommand{\p}{\Bbb{P}}

\newcommand{\N}{\mbox{\rm I\hspace{-0.02in}N}}
\newcommand{\R}{\mathbb{R}}

\newcommand{\ud}{\mathrm{d}}

\newcommand{\E}{\ensuremath{\mathbb{E}}}

\begin{document}

\title{The Wright-Fisher model with efficiency}
\author[1]{Adri\'an Gonz\'alez Casanova}
\author[2]{Ver\'onica Mir\'o Pina}
\author[3]{Juan Carlos Pardo}

\affil[1]{\footnotesize Instituto de Matem\'aticas de la Universidad Nacional Aut\'onoma de M\'exico, \'Area de la Investigaci\'on Cient\'ifica, Circuito Exterior, C.U., 04510 Coyoac\'an, CDMX, M\'exico. \textit{ Corresponding author:} {adriangcs@matem.unam.mx }} 
\affil[2]{\footnotesize Instituto de Investigaciones en Matem\'aticas Aplicadas y Sistemas, Universidad Nacional Aut\'onoma de M\'exico, Circuito Escolar 3000, C.U., 04510 Coyoac\'an, CDMX, M\'exico}
\affil[3]{\footnotesize Centro de Investigaci\'on en Matem\'aticas A.C., Calle Jalisco s/n. 36240 Guanajuato, M\'exico}

\date{}
\maketitle 
\begin{abstract}
In populations competing for resources, it is natural to ask whether consuming fewer resources provides any selective advantage. 
To answer this question, we propose a Wright-Fisher model with two types of individuals: the inefficient individuals, those who need more resources to reproduce and can have a higher growth rate,  and the efficient individuals. In this model, the total amount of resource $N$, is fixed, and the population size varies randomly depending on the number of efficient individuals. We show that, as $N$ increases, the frequency process of efficient individuals converges to a diffusion which is a generalisation of the Wright-Fisher diffusion with selection. 
The genealogy of this model is given by a branching-coalescing process that  we call the \textit{Ancestral Selection/Efficiency Graph}, and that is an extension of the Ancestral Selection Graph (\cite{KN1, KN2}).
The main contribution of this paper is that, in evolving populations, inefficiency can arise as a promoter of selective advantage and not necessarily as a trade-off.
\end{abstract}



\section{Introduction}

\subsection{Biological context}
 Organisms within an ecological niche compete for resources such as food or water. Different populations competing for the same resource can display different consumption strategies. While some organisms use small amounts of it to reproduce, their competitors may need larger amounts. Here we classify these two resource counsumption strategies as \textit{efficient}, when the cost of producing progeny is low, and \textit{inefficient} when this cost is high. 
Efficient individuals may have some advantage, as they can have more progeny when the resource is scarce. Nonetheless, this is also altruistic, since untapped resources can be used by their competitors. 
Therefore, it is not clear how  these different strategies are selected, or how do efficiency and cost shape the evolution of populations.

 For instance, one can think of the long term evolution experiment with {\it Escherichia coli}, led by Richard Lenski (see \cite{Lenski, Lenskitravisano} for an overview and \cite{GKWY, baake} for a mathematical model). In the twelve parallel replicates, after 10000 generations bacteria have become bigger. This suggests that consuming more resources could provide some advantage (see \cite{Lenskitravisano} for further details). 

%On the other hand, having an inefficient strategy could also be advantageous, as consuming as many resources as possible reduces the amount of resource available for competitors.

%Living organisms may compete for resources such as food or water. Some of them have more \textit{efficient} consumption strategies, in the sense that they need fewer resources to survive and reproduce than their competitors.
%One might think that those individuals have some advantage, as they can have, in principle, more progeny. Nonetheless, being efficient is also an altruistic strategy since untapped resources can be used by the (inefficient) competitors. 
%Therefore, it is natural to ask: {\it how are differences in resource consumption strategies selected?} 
%or {\it how do they shape the evolution of populations?}  
%(inefficient strategy). 
%On the other hand, having an inefficient strategy could also be  advantageous, as consuming as many resources as possible reduces the amount of resource available for competitors.
% %For instance, in the ongoing long term evolution experiment with {\it Escherichia coli}, led by Richard Lenski (see \cite{Lenski, Lenskitravisano} for an overview and \cite{GKWY, baake} for a mathematical model), bacteria that have been competing for thousands of generations have become bigger, raising the question whether consuming more resources  provides an advantage  per-se (see \cite{Lenskitravisano} for further details). 

Resource consumption strategies have been studied by ecologists. 
The $r/K$ theory (\cite{macarthur, pianka}) predicts that selective pressures will drive the evolution of a species into one of two general directions: $r$ strategies where the resource is allocated to the production of many low cost offspring; or $K$ strategies characterized by the investment of large amounts of resource in each descendant.
In that model the probability that an offspring survives to a reproductive age increases with its cost.
%The $r/K$ theory (\cite{macarthur, pianka}) predicts that selective pressures will drive the evolution of a species into one of two general directions: $r$ strategies where the resource is allocated to the production of many low cost offspring, with a low probability of surviving to a reproductive age; or $K$ strategies characterized by the investment of large amounts of resource in each descendant, increasing their probability of reaching a reproductive age. 
 %Our model allows us to study separately cost and probability of survival to a reproductive age. 
 %, such as Tilman's $R^*$ rule (\cite{tilman}) or the $r/K$ theory (\cite{macarthur, pianka}).
%Tilman's $R^*$ rule predicts that if multiple species are competing for a limiting resource, the only one that will survive is the one that has the most inefficient strategy (i.e. the highest resource consumption rate). 
\cite{tilman} proposed a model in which different species compete for a single resource (see  \cite{Miller} for a review). Species's growth rates are proportional to resource availability and their consumption rate. This theory predicts that the only species that will survive is the one that has the highest consumption rate (lowest $R^*$).
Here, we will define a model in which the cost of producing new offspring does not depend on resource availability and is not necessarily proportional to the reproduction rate (or the survival probability). 

Moreover, physiological  trade-offs have long been used to explain the persistence of inefficient microbial populations (see  \cite{molenaar} or \cite{lipson} and the references therein).
 %A trade-off between growth rate and efficiency can be understood in the mechanical analogy (see \cite{lipson}): racing cars get lower gas yield than slower, less powerful, fuel-efficient models.  
%which would tend to create two different strategies: fast-growing but inefficient and slow-growing but efficient.  
For example, in \cite{novak}, the authors have found a within-population negative correlation between growth and resource consumption rates, using bacterial populations from Lenski's experiment (\cite{Lenski}). 
 The molecular bases for these trade-offs are not so clear, although different metabolic explanations have been suggested (see \cite{beardmore} for a review). In our model, having an efficient strategy does not necessarily imply small growth rates.
%On the other hand, physiological trade-offs between growth rate and cost, in terms of resource consumption, have long been used to explain the persistence of inefficient microbial populations (see for example \cite{molenaar, lipson} or \cite{beardmore}). 
%From a physiological point of view,  trade-offs have long been used to explain the persistence of inefficient populations (see for example \cite{molenaar} or \cite{lipson}).
 %A trade-off between growth rate and efficiency can be understood in the mechanical analogy (see \cite{lipson}): racing cars get lower gas yield than slower, less powerful, fuel-efficient models.  
%which would tend to create two different strategies: fast-growing but inefficient and slow-growing but efficient.  
%For example, in \cite{novak}, the authors have found a within-population negative correlation between growth rate and yield, using bacterial populations from Lenski's experiment (\cite{Lenski}). 
% The molecular bases for these trade-offs are not so clear, although different metabolic explanations have been suggested (see for example \cite{beardmore} for a review).
 
 In this paper, we aim to study resource consumption strategies using population genetics arguments.
We propose an extension of the Wright-Fisher model (\cite{F58,W31})  where two types of individuals need different amounts of resource to produce offspring.
The total amount of resource at each generation is fixed but the population size is not: the higher the proportion of efficient individuals, the larger the population can be.
Models where the population size fluctuates stochastically had already been considered early in the population genetics literature, for example in \cite{seneta, donnelly, Tavare, Jagers, KK}.
%\textcolor{blue}{In the logistic branching process (\cite{Lambert2005}), individuals are subject to competition and, under some hypothesis on the offspring distribution, the population size remains finite but fluctuates randomly. 
Demographic stochasticity has also been modelled using birth-death processes, for example in \cite{Lambert2005} or  \cite{Parsons2007a, Parsons2007b, Parsons2008, Parsons2010}.
%In \cite{Lambert2005}, the author constructs a branching process in which, under some hypothesis on the offspring distribution, the population size  remains finite, due to competition, but fluctuates randomly. 
%In \cite{Parsons2007a, Parsons2007b, Parsons2008, Parsons2010}, the authors show that demographic stochasticity can affect important quantities, such a the fitness conferred by different traits or the time to fixation.
The specificity of our approach is that in our model there is a strong coupling between the size of the population and its genetic profile.

The main results of this paper is that inefficiency can enhance the effect of beneficial mutations, i.e.
if a beneficial mutation arises in an inefficient individual, it is more likely to go to fixation than if it arises in an efficient individual.
% if an inefficient population has a higher growth rate than the base population, it is more likely to go to fixation than a population that has the same (high) growth rate but is as efficient as the base population. In other words, inefficiency enhances the effect of having a higher growth rate on the fixation probabilities. 
Furthermore, differences in resource consumption strategies can be a mechanism of balancing selection.
 %This work could help bridging the gap between classical models in population genetics, and more explicit ecological and physiological assumptions.
% These effects occur whether the benefit of the mutation carried by the inefficient individuals is due to a physiological trade-off or is an independent mutation (maybe on a different gene).   
 This work could help bridging the gap between classical models in population genetics, and more explicit ecological and physiological assumptions.

 \subsection*{Outline}
 In Section \ref{themodel} we introduce the model in detail. 
 % In Sections \ref{S13},  \ref{introASEG} and \ref{S15} we present the main results, which are discussed  in Section \ref{S16}. 
  In Section \ref{S13}, we study the large population limit of the first version of the model and some of its properties such as the fixation probabilities.  In Section \ref{introASEG}, we study the associated genealogical process. Finally, Section \ref{S15} is devoted to the study of a second version of the model. These results are discussed in Section \ref{S16}. 
 Sections \ref{S2}, \ref{sectASEG} and \ref{S4} are  devoted to the proofs.

\subsection{The model}
\label{themodel}
The model is a modification of the classical Wright-Fisher model, in the sense that individuals still choose their parents  independently  at random from the previous generation. Instead of assuming a fixed number of individuals, we fix the amount of resource per generation.

In this section,  we fix $N\in\mathbb{N}$ and $ \kappa,s,x \in[0,1]$, where $N$ is the fixed amount of resource in each generation,  $\kappa$ denotes the efficiency parameter, $-s$ is the selection coefficient of the efficient individuals and $x$ parametrises  the initial frequency of efficient individuals.
We consider two types of individuals with different consumption strategies:
 \begin{itemize}
 \item type $0$ (efficient), that have selection coefficient $-s$ and need $1-\kappa$ units of resource to be produced.
  \item type $1$ (inefficient), that have selection coefficient 0  and need $1$ unit of resource to be produced.
 \end{itemize}
 %We will refer to type 0 individuals as ``efficient'', since their strategy consists in using a smaller amount of resource, and to type 1 individuals as ``inefficient''.
 The case where $s = 0$ corresponds to the neutral setting.
 %(the two types only differ by their consumption strategy). 
In the model with selection ($s>0$), the efficient individuals carry a deleterious mutation (or, equivalently, the inefficient individuals carry a beneficial mutation). Recall that the selective disadvantage of the efficient individuals can be due to a trade-off between efficiency and growth rate or to an independent mutation.
 The case $\kappa = 0$ corresponds to the classical Wright-Fisher model (with selection).

 In our model, generations are constructed recursively: at each generation we start with $N$ units of resource. The first individual consumes either 1 or $1-\kappa$ units of resource, implying that $N-1$ or $N-1+\kappa$ units are left for the rest of the individuals which are created using the same procedure.  We consider two rules to stop this procedure that lead to different behaviours. The model is defined as follows:
%In the neutral setting, given that the frequency of efficient individuals in generation $n-1$ is $x$, the first individual in generation $n$ is efficient with probability $x$ and inefficient otherwise. Depending on its type, the new individual consumes either 1 or $1-\kappa$ units of nutrients, implying that $N-1$ or $N-1+\kappa$ units of nutrients are left for the remaining individuals. The rest of the individuals are created using the same procedure.  We consider two rules to stop this procedure, that lead to different behaviours.  In both cases, reproduction is stopped when the amount of consumed resource is larger than $N$. In the first scenario, we keep all individuals while, in the second one, the last individual is not created if  there were fewer resources available than necessary to produce it.  Under both stopping rules, we also consider the case in which efficient individuals have  a negative selection coefficient ($s>0$). This means that efficient individuals are less likely to be chosen as parents of the new individual. Formally, we describe the  Wright-Fisher model with efficiency as follows.
\begin{definition}[Wright-Fisher model with efficiency parametrised by $N, \kappa, s$ and $x$]\label{WFWE}  
Each generation is a collection of individuals $\{1, \dots, M_n\}$, where $M_n$ is the population size at generation $n$. Let ${\tt t}(n,i)\in\{0,1\}$ be the type of the i-th individual of generation $n$ and 
\[
X^{(N)}_n:=\frac{1}{M_n}\sum_{j=1}^{M_n}\mathbf{1}_{\{t(n,j)=0\}}\qquad \textrm{and}\qquad C_{(n+1, i)}=i-\kappa\sum_{j=1}^i\mathbf{1}_{\{t(n+1,j)=0\}}, 
\]
be the frequency of efficient individuals at the $n$-th generation and the cost of producing the first $i$ individuals at  generation $n+1$, respectively. 
%Let  $N\in\mathbb{N}$ and $ \kappa,s,x \in[0,1]$, where $N$ is the fixed amount of resource in each generation,  $\kappa$ denotes the efficiency parameter, $-s$ is the selection coefficient of the efficient individuals and $x$ parametrises  the initial frequency of efficient individuals. %An individual is a vertex $v=(n,i)$, where $n\in\overline{\mathbb{N}}:=\mathbb{N}\cup \{0\}$ and $i\in\mathbb{N}$, 
%Generation $n\in\overline{\mathbb{N}}$ is the set of individuals with first coordinate $n.$ 

The initial condition is given by
$M_0 = N_x=(1-\kappa x)^{-1}N,$ which is the solution of 
\begin{equation}\label{eqN}
N=(1-\kappa)xN_x+(1-x)N_x, 
\end{equation}
and $X_0 = \lfloor xN_x\rfloor / N_x$.

 Individuals in generation $n+1$ are created recursively by one of the following rules. If $C_{(n+1,i)}<N$, the $(i+1)$-th individual is produced and either
\begin{description}
\item[-] if $s = 0$, she choses her parent uniformly at random from the previous generation  or
\item[-] if $s>0$, she choses a type 0 individual as its parent with probability
$$
\frac{(1-s)X_n^{(N)}}{1-sX_n^{(N)}}.
$$
and a type 1 parent otherwise.
\end{description}
The new individual copies the type of her parent. If $C_{(n+1,i)}\geq N$ then
\begin{description}
\item[(M1)]  $M_{n+1}=i$ (no more individuals are created) or
%\item[(M2)]  If $C_{(n+1,i)}> N$ then $M_{n+1}=i-1$ and the individual $(n+1,i)$ is discarded. In other words, the attempt of producing individual $(n+1,i)$ consumes all the remaining resources but the resources are insufficient to produce individual $(n+1,i)$.
\item[(M2)]  if $C_{(n+1,i)} = N$ then $M_{n+1}=i$ (no more individuals are created) and if $C_{(n+1,i)}> N$ then $M_{n+1}=i-1$ and  individual $(n+1,i)$ is discarded. In other words,  individual $(n+1, i)$,  is discarded if the remaining resources are insufficient to produce it.
\end{description}
\end{definition}
%The model defined under assumption {\rm {\bf (M1)}} is called the Wright-Fisher model with efficiency. The model {\rm {\bf (M2)}} is called the Wright-Fisher model with efficiency and exclusion, since  the last individual can be excluded depending on its type.  In both cases, depending on the parental selection rule, we add the adjective neutral, if $s = 0$,  or with selection, if $s>0$. 

\begin{figure}[h]
\begin{center}
\includegraphics[width=0.8\textwidth]{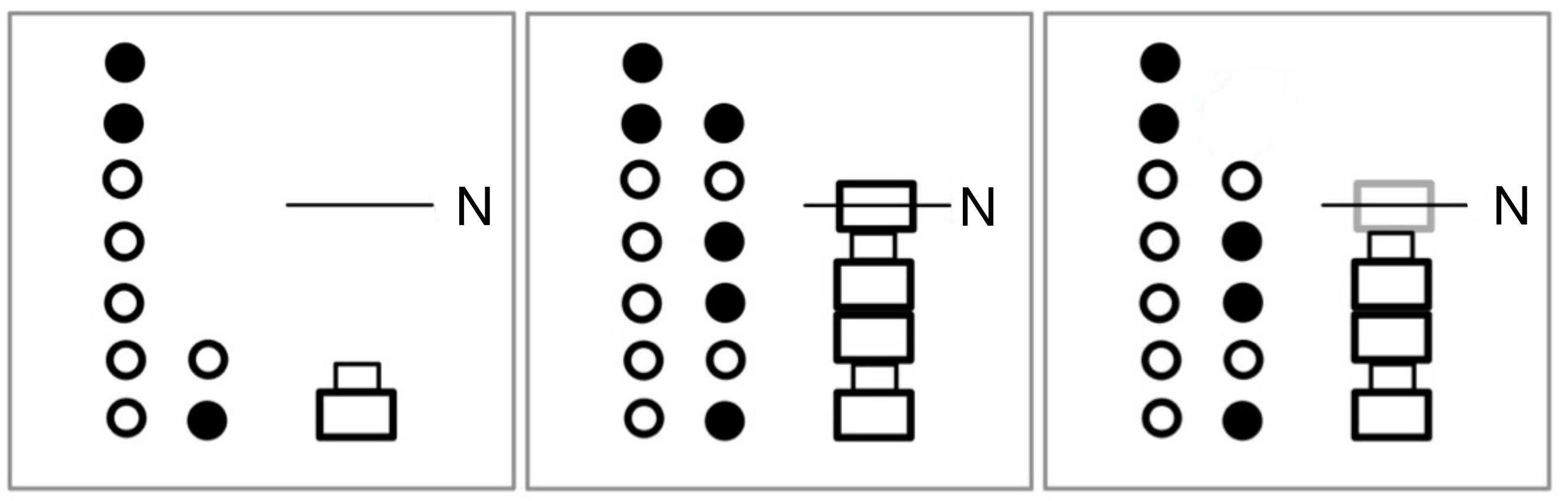}
\caption{\small  Example of how generation 1 is created in the Wright-Fisher model with efficiency with $s = 0$. Black (resp.  white) dots correspond to inefficient (resp. efficient) individuals. Generation 0 is represented by the leftmost column: $\lfloor xN_x \rfloor=5$ and $\lfloor N_x \rfloor=7$.
%, so each new individual in generation $1$ is efficient (white or type $0$) with probability $5/7$. 
The small (resp. large) rectangles represent the amount of resource consumed to produce an efficient (resp. inefficient) individual and have height 1 (resp. $1-\kappa$). The horizontal line represents the total amount of resource, $N$. The second panel corresponds to {\rm {\bf (M1)}} and $M_1 = 6$. The third panel, corresponds to {\rm {\bf (M2)}} and $M_1 = 5$ (the last individual is discarded). }
\label{Fig1}
\end{center}
\end{figure}

Figure \ref{Fig1} shows how generation 1 is created in the Wright-Fisher model with efficiency {\rm {\bf (M1)} or {\rm {\bf (M2)} and Figure \ref{simu1} shows a simulation of the model. 
There is a strong coupling between the population size and the frequency of efficient individuals: when $X_n= 1$, $M_n = \lfloor N/(1-\kappa)\rfloor + 1$ under {\rm {\bf (M1)}} ( $M_n = \lfloor N/(1-\kappa)\rfloor $ under {\rm {\bf (M2)}}) and when $X_n =0$, $M_n = N$.

%Under  {\rm {\bf (M1)}}, the population size fluctuates between $\lfloor N/(1-\kappa)\rfloor + 1$ (when there are only efficient individuals) and $N$ (when there are only inefficient individuals). Under the stopping rule {\rm {\bf (M2)}}, the population size fluctuates between $\lfloor N/(1-\kappa)\rfloor$ and $N$ .
% When $N\to \infty$, under both stopping rules, if $X^{N}_g = x$, the population size at generation $g+1$  remains close to $N_x$, which is  its expected value given $x$ (see Proposition \ref{LLN} in Section \ref{S2}).
   %Recall that, when $\kappa =1$ both stopping rules are the same, since, in generation $g$, we stop producing new individuals when we have produced exactly $N$ inefficient individuals (the efficient individuals do not contribute to the cost $C_{(g, i)}$).
The main difference between the two stopping rules lays in the fact that, under {\rm {\bf (M1)}}, we can always create efficient or inefficient individuals. But, under  {\rm {\bf (M2)}}, if the amount of consumed resource is in $(N-1, N]$ we can only produce efficient individuals. Therefore, under the stopping rule {\rm {\bf (M2)}}, being efficient could be advantageous (see Section \ref{S15}).

\begin{figure}[h]
\begin{center}
\includegraphics[width=0.99\textwidth]{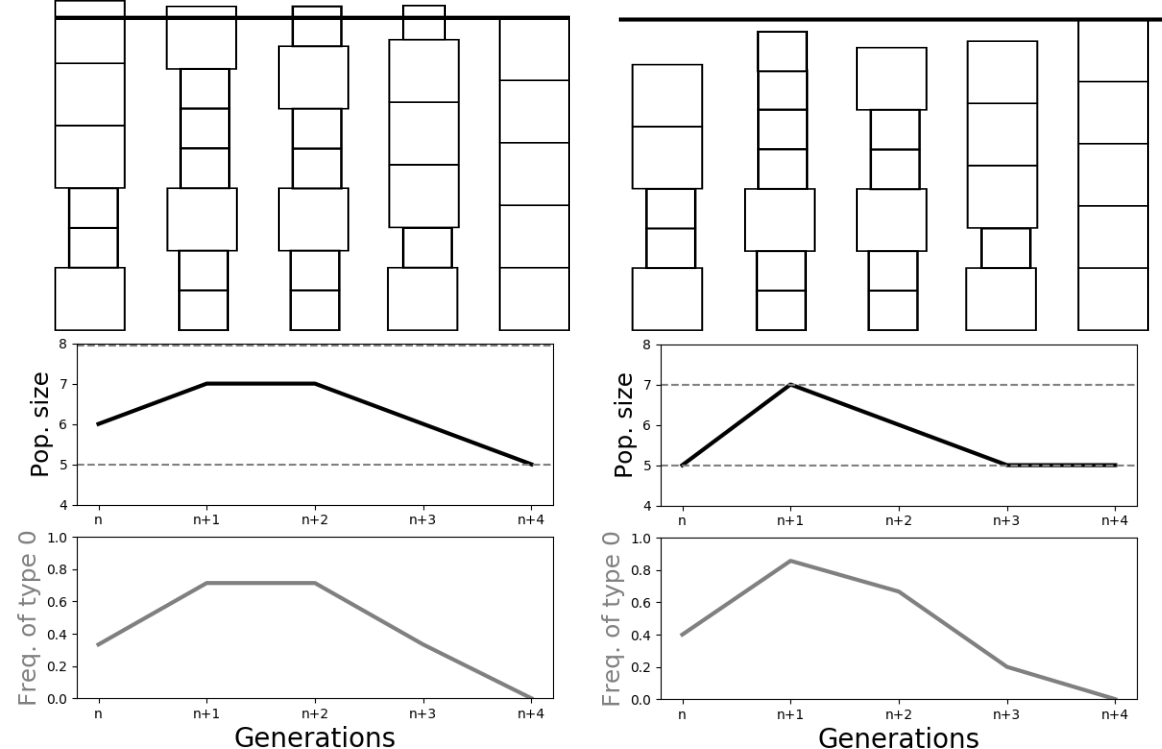}
\caption{\small  Simulations of the model ({\rm {\bf (M1)}}: left side panels and  {\rm {\bf (M2)}}: right side panels), for $N = 5$,  $\kappa = 0.3$ and $s = 0.1$. In the upper panels,  the large (resp. small) rectangles represent the amount of resource consumed to produce an  inefficient (resp. efficient) individual and the horizontal line has height $N$.  The bottom panels represent $(M_n, n\ge0)$ and $(X^{(5)}_n, n\ge 0)$ respectively. The dashed lines correspond to the bounds between which the population size can fluctuate.}
%: the lower bound is always $N = 5$ and the upper bound is $\lfloor N/(1-\kappa)\rfloor + 1 = 8$ for  {\rm {\bf (M1)}}  and $\lfloor N/(1-\kappa)\rfloor = 7$ for  {\rm {\bf (M2)}}. }
\label{simu1}
\end{center}
\end{figure}

\subsection{A diffusion-approximation under assumption {\rm {\bf (M1)}}  }
\label{S13}
Our first result (which will be proven is Section \ref{S2}) provides a scaling-limit for the Wright-Fisher model, when the population size is large and time is measured in the evolutionary scale (i.e. time is re-scaled by $N$). 

%In  the Wright-Fisher model with efficiency under assumption {\rm {\bf (M1)}},  with neutral ($s = 0$) or with  selective ($s>0$) parental rule,  the frequency of efficient individuals in the evolutionary scale is well approximated by a diffusion, which is the unique strong solution of a  stochastic differential equation (SDE for short). More precisely, we have the following result,  .

\begin{theorem}\label{thmscaling}
Fix $\kappa\in[0,1]$ and $\alpha\geq 0$. Let us consider a sequence of processes $\{(X^{(N)}_{\lfloor Nt \rfloor}, t\ge 0), N\ge 1\}$, as in Definition \ref{WFWE} under  {\rm {\bf (M1)}} and with neutral (if $s_N = \alpha/N=0$), or selective parental rule with selection coefficient  $s_N=\alpha/N>0$, and such that $X^{(N)}_{0}$ converges towards $x\in(0,1)$ in distribution.
Then the sequence  $\{(X^{(N)}_{\lfloor Nt \rfloor}, t\ge 0), N\ge 1\}$ converges weakly in the Skorokhod sense to the unique strong solution of the following stochastic differential equation (SDE)
\begin{equation}\label{selection}
\ud X_t=-\alpha X_t(1-X_t)\ud t+\sqrt{ X_t(1-X_t)(1-\kappa X_t)}\ud B_t,
\end{equation}
where  $B$ denotes a standard Brownian motion and with initial condition $X_0=x$.\end{theorem}

The SDE \eqref{selection} already appeared in \cite{GPP}.
Such unique strong solution, we call it  the {\it Wright-Fisher diffusion with efficiency {\rm {\bf (M1)}}} and from now on, it is denoted by $(X_t, t\geq0)$. 
 Figure \ref{simuM1} shows some simulated trajectories of $(X_t, t\geq0)$ and $(X^{(N)}_{\lfloor Nt \rfloor}, t\ge 0)$.

To obtain the diffusion approximation, time was re-scaled by $N$ and we assumed that the selection coefficient $s_N$ is of order $1/N$ (as in the classical Wright-Fisher model with selection). However, no re-scaling of the efficiency parameter  $\kappa$ is needed.

%Recall that no assumption is not needed to obtain the convergence. 

%, where the authors showed that it has a unique moment dual (which is the branching process with interactions presented in  Section \ref{introASEG}). Here we justify the relevance of its study by showing that it emerges as the scaling limit of the frequency process of efficient individuals in our Wright-Fisher model with efficiency and we study some of its asymptotic properties. 

\begin{figure}[h]
\begin{center}
\includegraphics[width=1.0\textwidth]{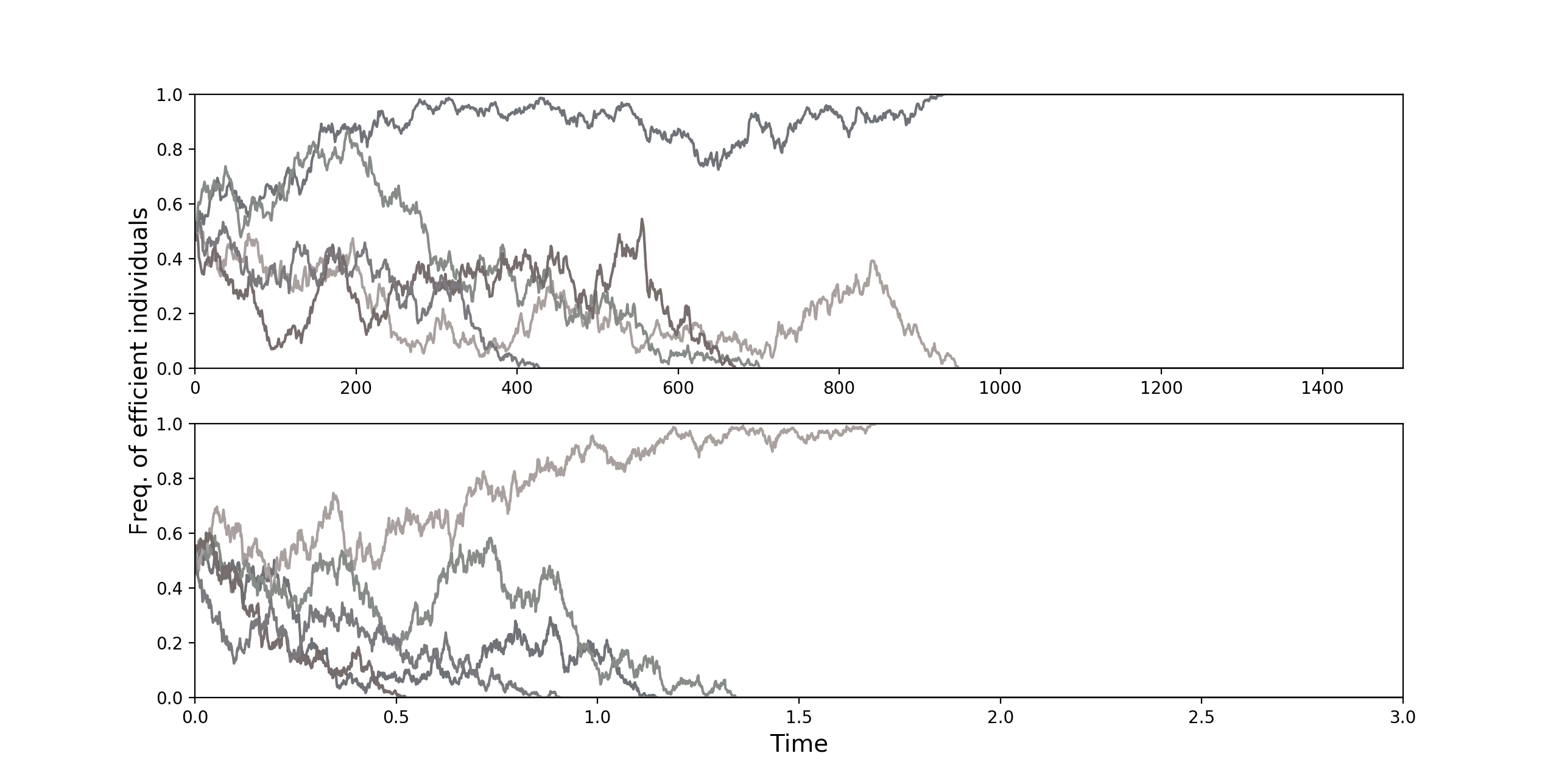}
\caption{\small The upper panel shows some simulations of Wright-Fisher model with efficiency and the bottom panel some trajectories of the Wright-Fisher diffusion with efficiency {\rm {\bf (M1)}}.  The parameters are $N = 500$, $x = 0.5$, $\kappa = 0.3$ and $\alpha = 0.1$. In the upper panel, time is expressed in generations. In the bottom panel, 1 time-unit corresponds to $N$ generations.}
\label{simuM1}
\end{center}
\end{figure}

The extra factor $(1-\kappa X_t)$ in the infinitesimal variance is the main contribution of efficiency. Roughly speaking this term appears since, when the frequency of efficient individuals is $x$, the population size remains close to $N_x = N(1-\kappa x)^{-1}$ (Proposition \ref{LLN} in Section \ref{S2})).
%and only appears in the diffusive term of \eqref{selection}. 
%Indeed, if $\kappa=0$, the classic Wright Fisher diffusion with selection is recovered.
% We will see that under assumption {\rm {\bf (M2)}} , the infinitesimal expectation of the diffusion is also affected by the difference of costs of the two subpopulations. 
%Roughly speaking, the extra factor  in the infinitesimal variance appears since the effective population size in this model depends on the frequency of efficient individuals. In fact, we will show that, if the frequency of efficient individuals is $x$, the population size remains close to $N_x = N(1-\kappa x)^{-1}$ (Proposition \ref{LLN} in Section \ref{S2})).
%In other words, if there are a lot of efficient individuals the population size will be close to $(1-\kappa)^{-1}N$, while if most of the individuals are inefficient the population size will be similar to $N$.

The diffusion $(X_t, t\geq0)$  has two boundaries, $0$ and $1$. We will show in Section \ref{S2} that when $\kappa \in [0,1)$ these boundaries are accessible (i.e. that they can be reached in finite time). We will also prove the following result.
\begin{proposition}\label{T} Let $X=(X_t, t\ge 0)$ be the unique strong solution of \eqref{selection} and define the time to fixation as $$T_{0,1}=\inf\{t\ge 0: X_t=0 \textrm{ or } 1\}.$$
Then, 
\begin{itemize}
\item[i)] if $\kappa<1$ and $\alpha=0$, the expected time to fixation is given by
\begin{equation}\label{Tfix0}
\mathbb{E}[T_{0,1} | X_0 = x]=2x\log \left(\frac{1-\kappa x}{x(1-\kappa)}\right)-\frac{2(1-x)}{1-\kappa}\log \left(\frac{1-x}{1-\kappa x}\right),
\end{equation}
\item[ii)] if $\alpha>0$ and $\kappa<1$, then 
$ \mathbb{E}[T_{0,1} | X_0 = x]<\infty,$ 
\item[iii)] if $\kappa=1$ then $ \mathbb{E}[T_{0,1} | X_0 = x]=\infty$.
\end{itemize}
\end{proposition}

When $\kappa \in [0,1)$ absorption occurs in finite time. 
If $\alpha = 0$, $(X_t, t\geq 0)$ is a martingale, so the probability of fixation of the efficient individuals is equal to their initial frequency (as in the classical Wright-Fisher model). In other words, there is no advantage or disadvantage at the population level in being efficient. 
But the time to absorption increases with $\kappa$. This is not surprising, since the presence of inefficient individuals reduces the population size. 

% But the right hand side of \eqref{Tfix0} is an increasing function of $\kappa$, which means that, in the presence of inefficient individuals, the time to absorption is shorter. This is not surprising, since the presence of inefficient individuals reduces the population size. 

 Let us denote by $(Y_t, t\geq0)$ the process $Y = 1-X$, which corresponds to the (limiting) frequency of inefficient individuals.
Define $\p_{y}\big(\{{\rm Fix.}\}\big)=\p(Y_{T_{0,1}}=1 | Y_0 = y)$,  the probability of fixation of the inefficient individuals starting from $y\in (0,1)$.
When $\alpha>0$ and $\kappa=0$,  it is well known (see for instance Lemma 5.7 in \cite{eth}) that the probability of fixation of type 1 individuals is given by
$$ \p_{y}\big(\{{\rm Fix.}\}\big)=\frac{1-\exp\{-2\alpha y\}}{1-\exp\{-2\alpha\}}.$$
 %When $\kappa \in (0,1)$ we can show that 
\begin{proposition}
The probability of fixation of the inefficient individuals in the Wright-Fisher diffusion with efficiency  {\rm {\bf (M1)}}, parametrised by $\kappa \in (0,1)$ and $\alpha\geq0$ is given by
\begin{eqnarray*}\p_{y}(\{{\rm Fix.}\}) %&=&\p(X_{T_{0,1}}=0|X_o = 1-y)\\
&=&\left\{ \begin{array}{ll}
1-C_{\kappa, \alpha}(1-(1-\kappa (1-y))^{1-\frac{2\alpha}{\kappa}}) & \textrm{if}\quad 2\alpha\ne\kappa,\\
1-C_{\kappa, \alpha}\ln \left(\frac{1}{1-\kappa (1-y)}\right)& \textrm{if} \quad 2\alpha=\kappa,\\
%1-C_{\kappa, \alpha}((1-\kappa (1-y))^{1-\frac{2\alpha}{\kappa}}-1) & \textrm{if}\quad 2\alpha>\kappa,
\end{array}
\right.
\end{eqnarray*}
where $C_{\kappa, \alpha}$ is such that
 $$
C_{\kappa, \alpha}^{-1}=\left\{ \begin{array}{ll}
1-(1-\kappa)^{1-\frac{2\alpha}{\kappa}} & \textrm{if}\quad 2\alpha\ne\kappa,\\
\ln \left(\frac{1}{1-\kappa}\right) & \textrm{if} \quad 2\alpha=\kappa.
%(1-\kappa)^{1-\frac{2\alpha}{\kappa}}-1 & \textrm{if}\quad 2\alpha>\kappa.
\end{array}
\right.
$$ 
\label{propo1new}
\end{proposition}

The proof of this Proposition can be found in Section \ref{S2}.
 \begin{figure}\begin{center}
\includegraphics[width=1.0\textwidth]{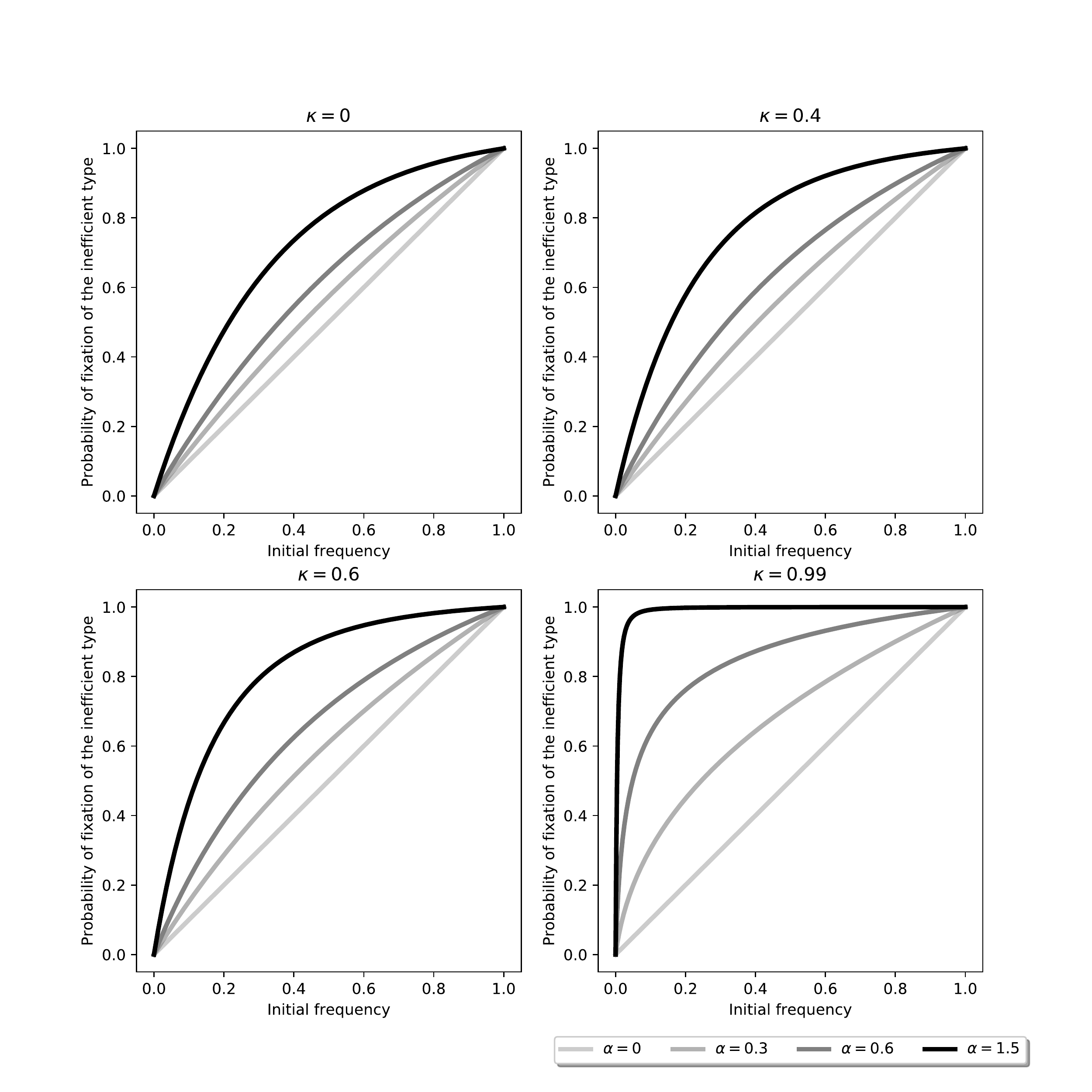}
\caption{\small Probability of fixation of inefficient individuals for the Wright-Fisher diffusion with efficiency (stopping rule  {\rm {\bf (M1)}}), for several combinations of  $\kappa$ and $\alpha$. }
\label{ffixation}
\end{center}
\end{figure}
As one can see in Figure \ref{ffixation}, for a fixed value of $\alpha$,  the probability of fixation of inefficient individuals increases when $\kappa$ increases. In other words, at the population level, being inefficient provides an advantage, since it increases the probability of fixation (compared to the classical Wright-Fisher model with the same selection coefficient).

\medskip

Finally, we consider the case $\kappa = 1$, which exhibits some interesting mathematical properties. It corresponds to the case where the cost of inefficient individuals is orders of magnitude larger than the cost of efficient individuals, in such a way that the efficiency parameter $\kappa_N$ tends to 1 as $N$ goes to infinity.
In that case, $0$ is accessible but $1$ is not.
However, it is still possible that  $X$ takes the value 1 in the limit. In other words, the efficient individuals might still ``go to fixation'' after an infinite time.
The first term of the right hand side of \eqref{Tfix0}, which is related to absorption at  $1$, goes to infinity when $\kappa$ goes to 1. But the second term, which is related to absorption in 0, converges to a finite value. 
%for every starting point $x\in [0,1)$.
% Indeed,
%\[\lim_{\kappa\rightarrow 1}-\frac{2(1-x)}{1-\kappa}\log \left(\frac{1-x}{1-\kappa x}\right)=2x.\]
The path behaviour is not easy to study using the classical theory of diffusions. Instead, we can use moment duality.

\subsection{A genealogical process associated to the Wright-Fisher diffusion with efficiency  {\rm {\bf (M1)}}}
\label{introASEG}
The {\it Ancestral Selection/Efficiency Graph}, that we introduce below, describes the genealogical structure associated to the Wright-Fisher model with efficiency  {\rm {\bf (M1)}}. It can be seen as an extension of the Ancestral Selection Graph (ASG) defined in \cite{KN1, KN2}.
%When we add efficiency into the Wright-Fisher model with selection (under assumption  {\rm {\bf (M1)}}), its associated scaling limit still has a notion of ancestry, at least in the sense of moment duality. 
%Indeed, the {\it Ancestral Selection/Efficiency Graph}, that we introduce below, describes the genealogical structure associated with the  Wright-Fisher model with efficiency. This genealogical process  can be seen as an extension of the Ancestral Selection Graph (ASG) introduced in \cite{KN1, KN2}.
%This result is surprising, since we can prove analytically that the ASEG is the moment dual of the Wright-Fisher diffusion with efficiency, but there does not seem to be a transparent mapping  between the  events of the two processes.
\begin{definition}\label{AEG1}
Fix $n\in \mathbb{N}, \kappa\in[0,1]$,  $\alpha \ge 0$. The random marked directed graph $\mathcal{G}^x_T$, with parameters $T>0$ and $x\in[0,1]$, that we call  the  Ancestral  Selection/Efficiency Graph (ASEG for short), is  a continuous-time Markov process that can be constructed as follows. Let $Z_t$ denote the number of active vertices at time $t$ and assume that $Z_0 = n$. If $Z_t = j$,  then 
\begin{enumerate}
\item[(i)] (Coalescence event) at rate $j(j-1)/2$ two uniformly chosen active vertices become inactive and produce a new active vertex which is connected to both of them,
\item[(ii)]  (Branching event)  at rate $\alpha j$ an uniformly chosen active vertex becomes inactive and produces two new active vertices which are connected to it,
\item[(iii)] (Pairwise branching event) at rate $\kappa j(j-1)/2$ an uniformly chosen active vertex becomes inactive and produces two new active vertices which are connected to it,
\item[(iv)] (Coloring the tips) at time $T>0$,  this procedure is stopped and each active vertex gets a type which is $0$ (efficient) with probability $x$ and $1$ (inefficient) otherwise, 
\item[(v)] (Coloring the inner vertices)  each inactive vertex is of type $0$ if and only if there is no directed path from it to any vertex of  type $1$. 
\end{enumerate}
\end{definition}

Figures \ref{figaseg} and \ref{AEG} represent the different types of events  $(i)-(iii)$ and a realization of the ASEG respectively.

\begin{figure}[h]
\begin{center}
\subfigure[coalescence]{\includegraphics[width=.14\textwidth]{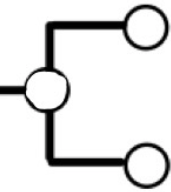}} \qquad \qquad
\subfigure[branching]{\includegraphics[width=.14\textwidth]{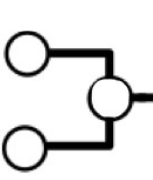}} \qquad \qquad
\subfigure[pairwise branching]{\includegraphics[width=.15\textwidth]{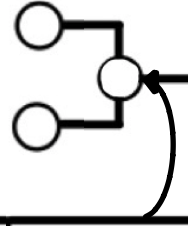}}
\caption{ \small Events of the ASEG. Pairwise branching events occur at rate proportional to $j(j-1)/2$, which corresponds to choosing a pair of vertices (connected by an arrow in this representation), but only one of them becomes inactive and produces two new vertices. }
\label{figaseg}
\end{center}
\end{figure}

The  {\it vertex counting process of the ASEG} $Z=(Z_t, t\ge 0)$ is a branching process with interactions,  with  parameters $(\alpha, 1, \kappa)$,  in the sense of \cite{GPP}. It is a birth-death process which goes from
% (Recall that, in that paper the ``pairwise branching'' events are referred to as ``cooperation'' events (they only occur if there are at least two active vertices, i.e. $j\ge2$). We avoid this term here since it can be confusing in our context and there is no cooperation in the biological sense). It is characterized by the following transition rates 
\begin{equation*}
 j\text{ to}\left\{ \begin{array}{ll}
j+1 & \textrm{ at rate } \alpha j+\kappa \frac{j(j-1)}{2},\\
j-1, & \textrm{ at rate } \frac{j(j-1)}{2}.
\end{array}
\right .
\end{equation*}

\begin{figure}[h]
\begin{center}
\includegraphics[width=.6\textwidth]{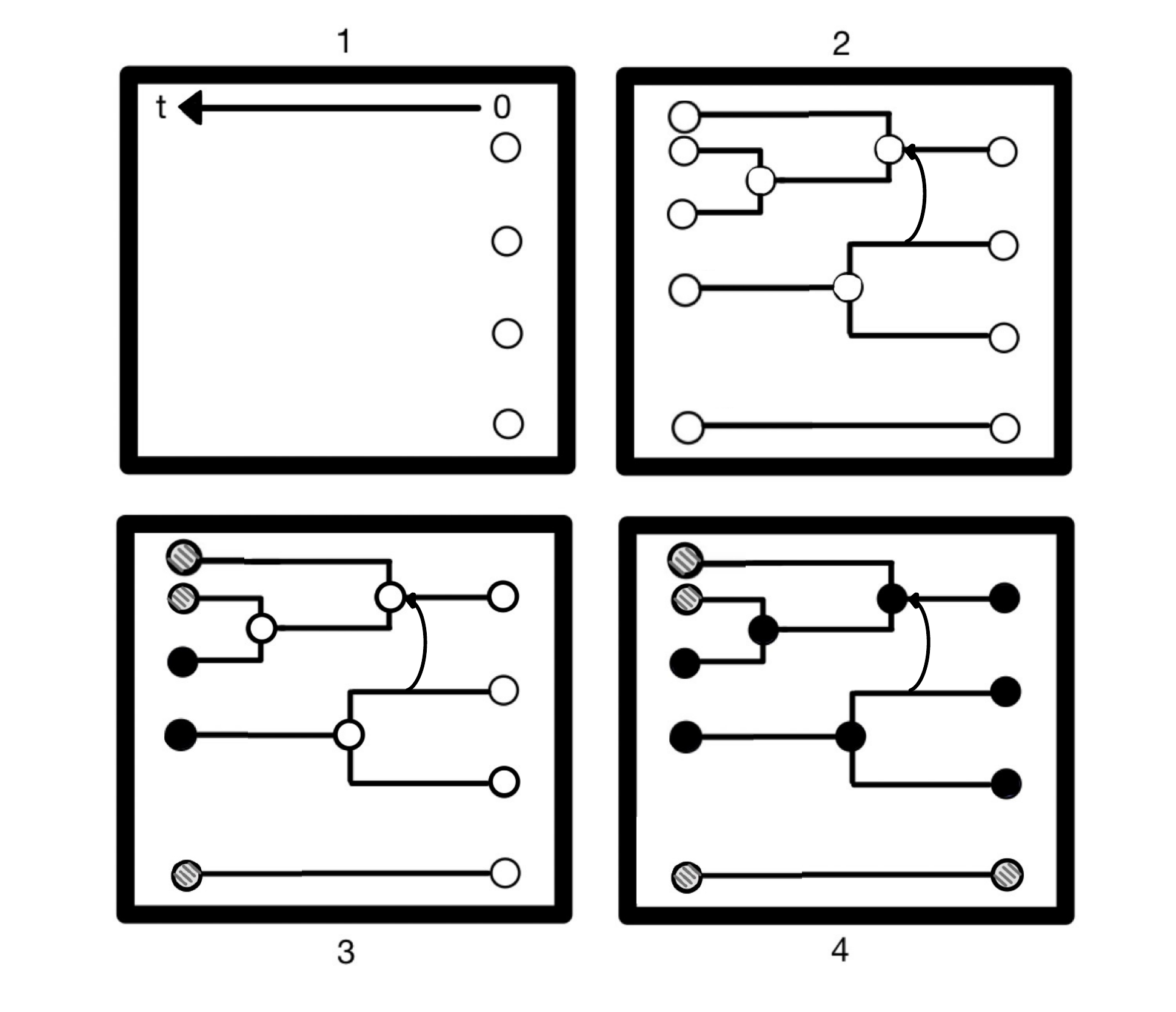}
\caption{ \small Realization of the ASEG. Panel 1 shows four active vertices at time 0. In panel 2, new vertices are produced according to rules $(i), (ii)$  and $(iii)$ of Definition \ref{AEG1}. In panel 3, external vertices are colored according to rule $(iv)$ (type 1 is represented in black and type 0 in grey). In panel 4,  individuals copy the type of their parents (the vertex to the left) following rule $(v)$.}
\label{AEG}
\end{center}
\end{figure}

%The interest of the previous construction is that for every $n\in\mathbb{N}$, $x$ and $T>0$ the probability that all vertices in $\mathcal{G}^x_T$ are of type 1 is the same as sampling $n$ independent Bernoulli random variables with parameter $X_T$, where $X_0=x$ and $X$ is the Wright-Fisher diffusion with efficiency. This statement is formalized in the next Lemma whose proof can be found in \cite{GPP}.
%The interest of the previous construction is that the ASEG is the genealogical process associated to the Wright-Fisher diffusion with efficiency.
\begin{lemma}%[Theorem 2 in \cite{GPP}]
\label{Lemduality}
The Wright-Fisher diffusion with efficiency $X$ defined as the unique strong solution of \eqref{selection}, with parameters $\alpha\ge 0$ and $\kappa \in [0,1]$  and the vertex counting process of the ASEG $Z$ defined above (with the same parameters),  are moment duals, i.e. for all $x\in[0,1]$, $n\in\mathbb{N}$ and $t>0$,
$$
\mathbb{E}_{x}[X_t^n]=\mathbf{E}_n[x^{Z_t}],
$$
where $\mathbb{E}_x$ and $\mathbf{E}_n$ denote the expectations associated to the laws of $X$ and $Z$ starting from $x$ and $n$, respectively. 
\end{lemma}
This result is a particular case of Theorem 2 in \cite{GPP}. For the sake of completeness, in Section \ref{sectASEG} we provide its proof. Intuitively, the left-hand side corresponds to the probability that $n$ independent Bernoulli random variables with parameter $X_T$ are equal to 1. The right-hand side is the probability generating function of $Z_t$ evaluated at the initial frequency $x$ and is therefore related to the probability that the vertices in  $\mathcal{G}^x_T$ are of type 1.
Lemma \ref{Lemduality} is surprising, since we can prove analytically that the ASEG is the moment dual of the Wright-Fisher diffusion with efficiency, but there does not seem to be a transparent mapping  between the  events of the two processes.
We refer the reader to Section \ref{sectASEG} for a more detailed study of the ASEG, which allows us to understand the path behaviour of the Wright-Fisher diffusion with efficiency  {\rm {\bf (M1)}}, when $\kappa = 1$.

\subsection{A diffusion-approximation under assumption {\rm {\bf (M2)}} }
\label{S15}
Finally, we study the scaling limit of the frequency process of efficient individuals under {\rm {\bf (M2)}}, when $\kappa$ is a rational number.  
  This assumption is interesting since it has natural interpretation: $\kappa=a/b$ means that $a$ inefficient individuals can be created with the amount of resource needed to produce $b$ efficient individuals. The case when $\kappa$ is irrational seems to be more involved.
  Note that, when  $\kappa=1$, both stopping rules {\rm {\bf (M1)}} and  {\rm {\bf (M2)}} are exactly the same since, in generation $g$, we stop producing new individuals when we have produced exactly $N$ inefficient individuals (the efficient individuals do not contribute to the cost $C_{(g, i)}$)
  
\begin{theorem}\label{scaling4}
Fix $\alpha \geq 0$  and  let $\kappa$ be a rational number in  $[0,1)$.
Let us consider a sequence of processes $\{(X^{(N)}_{\lfloor Nt \rfloor}, t\ge 0), N\ge 1\}$, as in Definition \ref{WFWE}, under  {\rm {\bf (M2)}} and with neutral   ($s_N = \alpha/N  = 0$) or selective  ($\alpha  > 0$)  parental rule, with selection coefficient $s_N = \alpha/N$ and such that $X^{(N)}_{0}$ converges towards $x\in(0,1)$ in distribution. Then the sequence  $\{(X^{(N)}_{\lfloor Nt \rfloor}, t\ge 0), N\ge 1\}$ converges weakly in the Skorokhod sense to the unique strong solution of one of the following SDE's:
\begin{itemize}
\item[(i)] if $\kappa=a/b\in(0,1/2)$ for some relative primes $a,b\in\mathbb{N}$, then 
\begin{equation}\label{diff4}
\ud X_t= (-\alpha + \kappa (1-\kappa X_t))X_t(1-X_t)\ud t+\sqrt{X_t(1-X_t)(1-\kappa X_t)}\ud B_t,
\end{equation}

\item[(ii)] if $1-\kappa=1/b\in(0,1)$ for some  $b\in\mathbb{N}$, then 
\begin{eqnarray}\label{diff5}
\ud X_t&=&\left( -\alpha + (1-\kappa) (1-\kappa X_t)\sum_{r=2}^{b-1} (1-X_t^r)\right) X_t(1-X_t) \ud t\nonumber \\ 
&&+\sqrt{X_t(1-X_t)(1-\kappa X_t)}\ud B_t,
\end{eqnarray}

\item[(iii)] if $1-\kappa=a/b\in(0,1/2)$ for some relative primes $a,b\in\mathbb{N}$, then  
\begin{eqnarray}\label{diff56}
\ud X_t&=&\left( -\alpha + (1-\kappa X_t)\sum_{i=1}^{m}c_i  (1-X_t^i)\right) X_t (1-X_t)\ud t \nonumber \\ && +\sqrt{X_t(1-X_t)(1-\kappa X_t)}\ud B_t,
\end{eqnarray}
where $m=\lfloor(1-\kappa)^{-1}\rfloor$, $c_m=1-ma/b$ and $c_i=a/b$ for all $i=1,2,...,m-1$.
\end{itemize}
In the three cases $B$ denotes a standard Brownian motion and the  initial condition is  $X_0=x$.
\end{theorem}
This result is proven in Section \ref{S4}.
 The unique strong solution of each of the above SDEs  is called  \textit{the Wright-Fisher diffusion with efficiency  {\rm {\bf (M2)}}} and denoted  by $(X_t, t\geq0)$.
 When $\alpha = 0$, the first term in the right-hand side of the three SDEs (the ``deterministic part'') is positive. As a consequence,  the efficient individuals have some ``advantage'', in the sense that their probability of fixation is higher than their initial frequency. 
 This phenomenon can be explained by the fact that when the amount of consumed resource is in $(N-1, N]$, only efficient individuals can be produced. 
 In the three cases, when $\alpha = 0$, the diffusions obtained can be interpreted as  random time-changes of some known diffusions. To be more precise,  in all cases above we have
 \[
 X_t=\tilde X_{\int_0^t (1-\kappa X_s)\ud s}, \qquad  0\le t\le T_{0,1},
 \]
 where the $\tilde X$ is the Wright-Fisher diffusion with selection (i.e $\kappa =0$) in the case $(i)$,  or the Wright-Fisher diffusion with frequency dependent selection (see equation (1) of \cite{GS}) in the cases $(ii)$ and $(iii)$. 
%In other words, since the random time change is bounded and the diffusions $Y$ are well known, the corresponding boundary problems can be solved using such random time change.

  \begin{figure}\begin{center}
\includegraphics[width=1.0\textwidth]{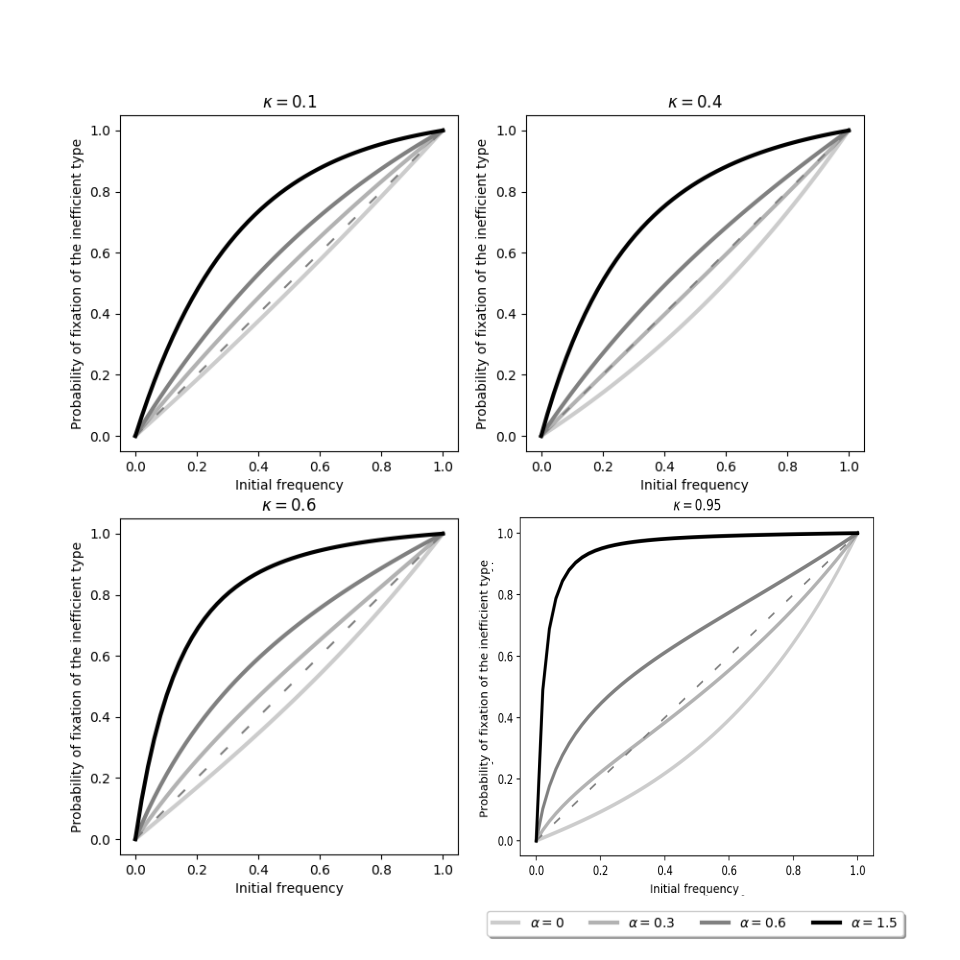}
\caption{\small Probability of fixation of inefficient individuals for the Wright-Fisher diffusion with efficiency {\rm {\bf (M2)}}, for several combinations of  $\kappa$ and $\alpha$. The dashed line corresponds to $y=x$, which is the probability of fixation in the neutral Wright-Fisher model. To compute the probability of fixation, the scale function was integrated numerically, using Scipy library for Python.}
\label{ffixation2}
\end{center}
\end{figure}
 When $s>0$, the sign of the first term in the three SDEs depends on the relative values of $\kappa, \alpha$ and $X_t$. 
  For some values of $\alpha$ and $\kappa$, the drift term is negative when $X_t$ is close to 1  and positive when $X_t$ is close to 0. This phenomenon is called \textit{balancing selection} (as selection ``pushes'' the frequency process towards intermediary values).
  % and it is considered to be one of  the most powerful evolutionary forces maintaining polymorphism (see for example \cite{turelli, fitzpatrick} or \cite{brisson} for a more recent review).  
 
For the sake of brievity, we now focus on the case $(i)$. We prove in Section \ref{S4} that both boundaries $0$ and $1$ are accessible, that the expected time to fixation is finite  and that we have the following result.
 \begin{proposition}
 The probability of fixation of the inefficient individuals in the Wright-Fisher diffusion with efficiency  {\rm {\bf (M2)}}  parametrised by  $\kappa=a/b\in(0,1/2)$ for some relative primes $a,b\in\mathbb{N}$ and  $\alpha\ge0$ is given by
$$\p_{y}(\{{\rm Fix.}\})= \frac{\int_{1-y}^1\exp(-2\kappa u) (1-\kappa u)^{-2\alpha/\kappa} \ud u}{\int_0^1\exp(-2\kappa u) (1 -\kappa u)^{-2\alpha/\kappa} \ud u}.$$ 
\label{propoM2}
\end{proposition}

As we can observe in Figure \ref{ffixation2}, for a fixed $\kappa$, if $\alpha$ is small, $\p_{y}(\{{\rm Fix.}\})< y$. The latter means that, at the population level, efficient individuals have some advantage (as in the case  {\rm {\bf (M2)}}, with $\alpha = 0$).
 However, when $\alpha$ is large enough (for example when $\kappa = 0.4$ and $\alpha = 1.5$),  $\p_{y}(\{{\rm Fix.}\})>y$, meaning that the inefficient individuals have some advantage at the population level,  as in the case {\rm {\bf (M1)}}.
 Finally, in some cases (for example $\kappa = 0.95$ and $\alpha = 0.3$ in Figure \ref{ffixation2}), we can see the effect of the balancing selection. Indeed, when we compare the fixation probabilities to the ones of the Wright-Fisher model with selection, if the initial frequency of inefficient individuals is low, they have some advantage. On the contrary, if their initial frequency is high, they have some disadvantage.
 Recall that, to observe this effect in a finite population of size $N$, we need a small selection coefficient ($s_N = \alpha/N$) compared to the efficiency parameter $\kappa$.

 \subsection{Discussion and open problems}
 \label{S16}
To summarize, we discuss the evolutionary consequences of efficiency by using an extension of the Wright-Fisher model in which different types of individuals need different amounts of resource to reproduce.
%n which there are two types of individuals:  the inefficient individuals, those who need more resources to reproduce and can have a higher growth rate,  and the efficient individuals.
%The population size varies randomly and depends on the frequency of efficient individuals: the higher the proportion of efficients, the more individuals can be created with the same (fixed) amount of resource $N$.
%We have shown that, when  $N$ is large, the frequency process of efficient individuals is well approximated by a diffusion that generalises the Wright-Fisher diffusion with selection (or with frequency dependent selection) by adding the term $(1-\kappa X_t)$ to its infinitesimal variance. This is due to the coupling between the frequency of efficient individuals and the population size (and thus the strength of the random genetic drift).
We consider two variations of the model, depending on the rule that is used to complete each generation. Under assumption {\rm {\bf (M2)}}, efficiency provides some advantage at the population level (see Theorem \ref{scaling4}). However, the main contribution of this work is that inefficiency can be part of the advantage of an emerging trait.  
%In our model, inefficiency does not necessarily indicate the existence of a trade-off (as it was suggested for example in \cite{molenaar, lipson}). It can be part of the advantage of an emerging trait. 
Imagine a situation in which two types of individuals with different strategies co-exist. A beneficial mutation is more likely to be fixed in the population if it arises in an inefficient individual. In other words, inefficiency acts as a promoter of selective advantage. 
This is always the case for rule {\rm {\bf (M1)}} and it is also true for rule  {\rm {\bf (M2)}} if the selection parameter is large enough (see Figures \ref{ffixation} and \ref{ffixation2}). 
This positive correlation between inefficiency and fixation probability occurs whether the beneficial mutation carried by the inefficient individuals is due to a physiological trade-off or is an independent mutation (maybe on a different gene). %In our model, inefficiency does not necessarily indicate the existence of a trade-off (as it was suggested for example in \cite{molenaar, lipson}). 

The Wright-Fisher diffusion with efficiency {\rm {\bf (M1)}} has a moment dual, which is a branching-coalescing process, the ASEG.
 This result, which generalises the ASG, comes from analytic manipulations, but still requires a transparent interpretation in terms of discrete models.
 It is possible to construct a Moran model with selection where individuals choose a certain number of potential parents and such that, if time is reverted, it contains the ASG (see \cite{Mano, Lenz, Kluth}).  However, it is not clear how we can extend this construction to the ASEG  and how to interprete the pairwise branching events. This is an interesting open question. 
In the case of assumption {\rm {\bf (M2)}}, we were not able to find an associated genealogical process. It is also an open question, possibly related to the first one.
%In fact, understanding the mapping between the events of the Wright-Fisher model with efficiency (under assumption  {\rm {\bf (M1)}}) and the events of the ASEG could shed some light into finding a genealogical process associated to the Wright-Fisher model with efficiency and exclusion (under assumption  {\rm {\bf (M2)}}). 

Finally, efficiency can act as a mechanism of balancing selection, which is considered to be one of  the most powerful evolutionary forces maintaining polymorphism (see for example \cite{turelli, fitzpatrick} or \cite{brisson} for a more recent review).  
 Many different mechanisms of balancing selection have been suggested in the literature, such as within niche competition, heterozygote advantage, self-incompatibility between mating types and host-parasite or prey-predator interactions (see \cite{brisson} and the references therein), but to our knowledge this is the first work in which it arises from within-species differences in resource consumption strategies.
The consequences of these effects still have to be studied especially by means of  experiments.

\section{Scaling limit of model {\rm {\bf (M1)}} }\label{S2}
In order to prove Theorem \ref{thmscaling}, we first deduce the following proposition which  is crucial for determining the scaling limits of the Wright-Fisher model with efficiency. In particular, it says that the total number of individuals in a generation $n+1$ is close to its expectation given the frequency of efficient individuals in  generation $n$.
\begin{proposition}\label{LLN}
In the Wright-Fisher model with efficiency, with neutral parental rule ($s = 0$), and with stopping rule {\rm {\bf (M1)}} or {\rm {\bf (M2)}}, parametrised by $\kappa\in [0,1]$,
given $X^{(N)}_n=x\in[0,1]$, for every  $a\in (-1/2,0)$, we have 
$$\lim_{N\rightarrow \infty}\inf_{x\in(0,1)} \mathbb{P}\left(\frac{M_{n+1}}{N_x}\in [1-N^a,1+N^a]\right)=1.$$
\end{proposition}
\begin{proof}Let $\{B_i; i\ge 1\}$ be a sequence of i.i.d.  Bernoulli random variables with parameter $x$. For every $a\in (-1/2,0)$, using the definition of $N_x$ \eqref{eqN}, we get
\[
\begin{split}
\mathbb{P}\Bigg(\frac{M_{n+1}}{N_x}\leq 1-N^a\Bigg)&=\mathbb{P}\Big(C_{(n+1, \lfloor N_x (1-N^a) \rfloor)}\  {\ge} \ N\Big)\\
&=\mathbb{P}\left(\lfloor N_x (1-N^a) \rfloor-\kappa\sum_{i=1}^{\lfloor N_x (1-N^a) \rfloor}B_i\  {\ge} \ N\right)\\
&\le\mathbb{P}\left(N_x\Big(\kappa x- N^{a}\Big) \  {\ge} \ \kappa\sum_{i=1}^{\lfloor N_x (1-N^a) \rfloor}B_i\right).
\end{split}
\]
Adding and subtracting  $\kappa xN_xN^a$, using \eqref{eqN} and then adding and subtracting  $\kappa x\lfloor N_x (1-N^a) \rfloor$ yields
\[
\begin{split}
%&=\mathbb{P}\left(kxN_x(1-N^a)- N_x(1-\kappa x)N^{a}>\kappa\sum_{j=1}^{\lfloor N_x (1-N^a) \rfloor}B_i\right)\\
&\mathbb{P}\Bigg(\frac{M_{n+1}}{N_x} \leq 1-N^a\Bigg)\le\mathbb{P}\left(\kappa xN_x(1-N^a)- N^{1+a}\  {\ge} \ \kappa\sum_{i=1}^{\lfloor N_x (1-N^a) \rfloor}B_i\right)\\
&=\mathbb{P}\left(- N^{1+a}+\kappa x\Big(N_x(1-N^a)-\lfloor N_x (1-N^a) \rfloor\Big)\  {\ge}\ \kappa\sum_{i=1}^{\lfloor N_x (1-N^a) \rfloor}B_i-\kappa x\lfloor N_x (1-N^a) \rfloor\right)\\
&\leq\mathbb{P}\left(- N^{1+a}+1 \  {\ge} \ \kappa\sum_{i=1}^{\lfloor N_x (1-N^a) \rfloor}(B_i-x)\right)\\
&\leq\mathbb{P}\left(\kappa\left|\sum_{i=1}^{\lfloor N_x (1-N^a) \rfloor}(B_i-x)\right| \  {\ge} \ N^{1+a}+1\right).
\end{split}
\]

 For the upper bound,  a similar strategy can be used to get 
  \[
\begin{split}
 \mathbb{P}\left(\frac{M_{n+1}}{N_x}\geq 1+N^a\right)&=\mathbb{P}\Big(C_{(n+1, \lfloor N_x (1+N^a) \rfloor )}<N\Big)\\
 &\le\mathbb{P}\left(N_x\Big(\kappa x+ N^{a}\Big)>-1+\kappa\sum_{i=1}^{\lfloor N_x (1+N^a) \rfloor}B_i\right)\\
 &\leq\mathbb{P}\left(\kappa\left|\sum_{i=1}^{\lfloor N_x (1+N^a) \rfloor}(B_i-x)\right|> N^{1+a}+2\right).
\end{split}
\]
Next, since $B_i-x$ is a centered random variable with $\mathbb{V}ar[B_i]=x(1-x)$, we have
\[
\begin{split}
\mathbb{V}ar\left[\sum_{i=1}^{\lfloor N_x (1-N^a) \rfloor}(B_i-x)\right]&\le \sum_{i=1}^{\lfloor N_x (1+N^a) \rfloor}\mathbb{V}ar[B_i-x]\\
&=\frac{N}{1-\kappa x} (1+N^a) x(1-x)\\
&\le N(1+N^a).
\end{split}
\]
In both cases, from  Tchebycheff's inequality we obtain
\[
\begin{split}
\lim_{N\rightarrow \infty}\sup_{x\in(0,1)} &\mathbb{P}\left(\frac{M_{n+1}}{N_x} \notin [1-N^a, 1+N^a]\right)\\
&\leq\lim_{N\rightarrow \infty}\sup_{x\in(0,1)} \mathbb{P}\left(\kappa\left|\sum_{i=1}^{\lfloor N_x (1-N^a) \rfloor}(B_i-x)\right|> N^{1+a}+1\right)\\ &\leq\lim_{N\rightarrow \infty}\frac{N(1+N^a)}{(N^{(1+a)}+1)^2}=\lim_{N\rightarrow \infty}N^{1-2(1+a)}=0,
\end{split}
\]
where we have used that $1-2(1+a)<0$ since $a\in(-1/2,0)$. The proof of this proposition now follows. 
\end{proof}
\begin{proof}[Proof of Theorem \ref{thmscaling}] Classical results for SDEs with H\"older continuous coefficients provide that  \eqref{selection} has a unique strong solution (see  for instance Theorem 2 in \cite{GPP}). 
 We denote by $\mathcal{A}$ the infinitesimal generator of the Wright-FIsher diffusion with efficiency. Its domain contains the set of continuously twice differentiable functions on $[0,1]$, here denoted by $C^2([0,1])$, and for every $f$ in $C^2([0,1])$, and for every $x \in [0,1]$, we have
 \begin{equation}\label{genX}
\mathcal{A}f(x) =-\alpha x(1-x)f'(x)+x(1-x)(1-\kappa x)f''(x).
\end{equation}

Recall that, given $X^{(N)}_0 = x$ and $M_1 = m$, $m X_1^{(N)}$ follows a binomial distribution with parameters $m$ and $(1-s_N)x/(1-s_N x)$,  using Taylor's expansions, as in the classical Wright-Fisher model with selection, we have 
\begin{eqnarray*}
N\mathbb{E}\Big[X_1^{(N)}-x | X_0^{(N)} = x, M_1 = m \Big] &=& - \alpha x(1-x) + o(1) \\
N\mathbb{E}\Big[(X_1^{(N)}-x )^2| X_0^{(N)} = x, M_1 = m \Big] &=&  \frac{N}{m} x(1-x) + o(1),
\end{eqnarray*}
where $N/m \le 1$.
If we take $f\in C^2([0,1])$, $x \in [0,1]$, $a\in(-1/2,0)$ and using Proposition \ref{LLN},  the discrete generator of $X^{(N)}_{\lfloor Nt \rfloor}$ satisfies
\[
\begin{split}
A^Nf(x):=&N\mathbb{E}\Big[f(X_1^{(N)})-f(x)| X_0^{(N)} = x\Big]\\
=&N\mathbb{E}\Big[X_1^{(N)}-x | X_n^{(N)} = x \Big]f'(x)+N\mathbb{E}\Big[(X_1^{(N)}-x)^2| X_0^{(N)} = x \Big]f''(x)+o(1)\\
=&-\alpha x(1-x)f'(x)\\
& + N\mathbb{E}\Big[(X_1^{(N)}-x)^2\Big|X_0^{(N)} = x, M_1\in[N_x(1-N^a),N_x(1+N^a)]\Big]f''(x)  +o(1)\\
%& + N\mathbb{E}\Big[(X_1^{(N)}-x)^2\Big|X_0^{(N)} = x, M_1\notin[N_x(1-N^a),N_x(1+N^a)]\Big] \\ &\ \p(M_1\notin[N_x(1-N^a),N_x(1+N^a))f''(x) +o(1) \\
=&-\alpha x(1-x)f'(x)+\frac{N}{\lfloor N_x \rfloor}x(1-x)f''(x)+o(1)\\
=&-\alpha x(1-x)f'(x)+x(1-x)(1-\kappa x)f''(x)+o(1) \underset{N \to \infty}{\longrightarrow}\mathcal{A}f(x),
\end{split}
\]
where the term $o(1)$   depends on $x$ but converges to 0 uniformly in $x$. 
Since all the processes involved are Feller taking values on $[0,1]$ and the convergence of the generators is uniform by Proposition \ref{LLN}, then  the result follows from Lemma 17.25 of \cite{Kallemberg}.
\end{proof}

The rest of this section is dedicated to the study of the Wright-Fisher diffusion with efficiency. We start by proving the following proposition. 
\begin{lemma}\label{propo1}Let $X=(X_t, t\ge 0)$ be the unique strong solution of \eqref{selection},   the following statements holds.
\begin{itemize}
\item [i)] For  $\alpha\ge 0$ and $\kappa<1$,  the boundary points $0$ and $1$ are accesible.

\item[ii)] For  $\kappa=1$ and  $\alpha\geq 0$, the boundary  $1$ is not accesible and the boundary $0$ is accesible.
\end{itemize}
\end{lemma}

\begin{proof}
We start by introducing the Wright-Fisher diffusion with selection  parameter $c \in \R$ as the unique strong solution of  
\begin{equation}\label{WFselection}
\ud \hat Y^{(c)}_t=c\hat Y^{(c)}_t\Big(1- \hat Y^{(c)}_t\Big)\ud t+\sqrt{ \hat Y^{(c)}_t\Big(1-\hat Y^{(c)}_t\Big)}\ud B_t,
\end{equation}
where  $B$ denotes a standard Brownian motion.

 We first deal with the case $\alpha\ge 0$ and $\kappa\in(0,1)$. To do so, we use  a stochastic domination argument.  Let us introduce the following diffusion $\overline{Y}^{(c)}=(\overline{Y}^{(c)}_t, t\ge 0)$ which is obtained as a random time change of $\hat Y^{(c)}$, that is to say,  for the clock
\[
A_t=\int_0^t\frac{\ud s}{1-\kappa \hat Y^{(c)}_s}, \qquad \textrm{for}\quad t\ge 0,
\]
we introduce $\overline{Y}^{(c)}_t= \hat Y^{(c)}_{\theta_t},$ for $ t\ge 0$ where $\theta_t=\inf\{u:A_u>t\}$ is the right-continuous inverse of the clock $A$. Using \eqref{WFselection}, we observe that  $\overline{Y}^{(c)}$ satisfies the following SDE
\[
\ud \overline{Y}^{(c)}_t=c\overline{Y}^{(c)}_t\Big(1-\overline{Y}^{(c)}_t\Big)\Big(1-\kappa \overline{Y}^{(c)}_t\Big)\ud t+\sqrt{ \overline{Y}^{(c)}_t\Big(1-\overline{Y}^{(c)}_t\Big)\Big(1-\kappa \overline{Y}^{(c)}_t\Big)}\ud \beta_t,
\]
where $\beta=(\beta_t, t\ge 0)$ is a standard Brownian motion.
Since, for every $t\ge 0$, we have 
\[
\theta_t=\int_0^t\Big(1-\kappa \overline{Y}^{(c)}_s\Big)\ud s\in [(1-\kappa)t,t],
\] 
and  $ Y^{(c)}$ goes to fixation in finite time a.s. (see for instance equation (3.6) in  \cite{ewen}), we deduce that this is also the case for  $\overline{Y}^{(c)}$. In other words, both boundaries are accessible  for $\overline{Y}^{(c)}$ for every $c\in \R$. 

Finally, since for any fixed $\alpha>0$, we have that a.s. 
\begin{equation*}
\overline{Y}^{(-\alpha)}_t\geq X_t \geq \overline{Y}^{(\alpha/(\kappa-1))}_t, \qquad t\ge 0,
\end{equation*}
and for $\alpha=0$,  $X \equiv \overline{Y}^{(0)}$, we conclude that the boundaries $\{0,1\}$ are also accessible for $X$.

\medskip

For the case $\alpha=0$ and $\kappa=1$, we use the following integral test  which says that the boundary $a$  is accessible if and only if
\begin{equation*}
\int_{x_0}^a M(u)\ud S(u)<\infty
\end{equation*}
where $M$ and $S$ denote the speed measure and  the scale function associated to $X$ (see for instance Chap. 8 in  \cite{EK}). In our case both functions can be computed explicitly. Indeed, the scale function is proportional to the identity and the speed measure satisfies
\begin{equation}\label{speed}
M(x)=\int_{x_0}^x \frac{1}{u(1-u)^2}\ud u=\frac{1}{1-x}-\log(1-x)+\log(x)+c(x_0),
\end{equation}
where $c(x_0)$ is a constant that only depends on $x_0$. We have
\begin{equation*}
\int_{x_0}^1\left( \frac{1}{1-u}-\log(1-u)+\log(u) \right)\ud u=\infty
\end{equation*}
while 
\begin{equation*}
\int_{0}^{x_0}\left( \frac{1}{1-u}-\log(1-u)+\log(u) \right)\ud u<\infty,
\end{equation*}
so 1 is not accessible and 0 is accessible.
The case  $\alpha>0$ and $\kappa=1$, follows directly from a stochastic domination argument by the neutral  Wright-Fisher diffusion with efficiency $\kappa=1$ (and $\alpha=0$) studied above. 
 \end{proof}

\begin{proof}[Proof of Proposition \ref{T}] We first deduce part $(i)$, i.e. we compute the expected time to fixation for the case $\alpha=0$ and $\kappa\in(0,1)$.  To do so,  we use Green's function (see for instance Theorem 3.19 in  \cite{eth}) i.e.
 \[
\mathbb{E}\Big[T_{0,1} | X_0  = x\Big]=\int_0^1 G(x,u)\ud u,
\]
 where the Green's function $G$ is such that
\begin{equation}
G(x,u)=\left\{ \begin{array}{ll}
\frac{2x}{u(1-\kappa u)}& \textrm{ for }  x<u<1 ,\\
 \textrm{  }\\
\frac{2(1-x)}{(1-u)(1-\kappa u)}& \textrm{ for }  0<u<x,
\end{array}
\right .
\label{green}
\end{equation}
implying that  the expected time to fixation satisfies
\[
\begin{split}
\mathbb{E}\Big[T_{0,1} | X_0  = x\Big]&=-2x\Big(\log(x)+\log(1-\kappa)-\log(1-\kappa x)\Big)\\
&\hspace{2cm}-2(1-x)\left(\frac{\log(1-x)}{1-\kappa}-\frac{\log(1-\kappa x)}{1-\kappa}\right).
\end{split}
\]
For part $(ii)$, i.e. when $\alpha>0$ and $\kappa\in(0,1]$,  the Green's function is such that for 
\[
G(x,u)=\left\{ \begin{array}{ll}
\frac{2(S(x)-S(0))}{(S(1)-S(0))S'(u)}\frac{S(1)-S(u)}{u(1-u)(1-\kappa u)}& \textrm{ for }  x<u<1 ,\\
 \textrm{  }\\
\frac{2(S(1)-S(x))}{(S(1)-S(0))S'(u)}\frac{S(u)-S(0)}{u(1-u)(1-\kappa u)}& \textrm{ for }  0<u<x,
\end{array}
\right .
\]
where $S$ is the scale function and satisfies for $x \in [0,1]$
\begin{equation}\label{S}
S(x)=\int_0^x \exp\left\{\frac{2\alpha}{\kappa}\int_\theta^u\frac{\ud v}{\frac{1}{\kappa}-v}\right\}\ud u \ =\left\{ \begin{array}{ll} K\left(1-\left(1-\kappa x\right)^{1-\frac{2\alpha}{\kappa}}\right)&  \textrm{for }  \kappa \ne 2\alpha \nonumber \\
 K\ln \left(\frac{1}{1-\kappa x}\right)& \textrm{for } \kappa= 2\alpha,
\end{array}
\right.
\end{equation}
 where $\theta$ is an arbitrary positive number and $K$ is a constant that depends on $( \kappa, \alpha, \theta)$.
To determine whether the Green's function is integrable on $[0,1]$, it is enough to  study its behavior near the boundaries 0 and 1.  Indeed, for $u$  close to 0, we have
\[
G(x,u)\sim\frac{2(S(1)-S(x))}{(S(1)-S(0))}\frac{1}{(1-u)(1-\kappa u)},
\]
so the function  $u\mapsto G(x,u)$ is always integrable in a neighborhood of 0.
Moreover,  when  $u$  is close to 1, we get
\[
G(x,u)\sim \frac{2(S(x)-S(0))}{(S(1)-S(0))}\frac{1}{u(1-\kappa u)},
\]
so the function  $u\mapsto G(x,u)$ is integrable in a neighborhood of 1 if $\kappa<1$ and is not integrable if $\kappa = 1$, which completes the proof. 
\end{proof}

\begin{proof}[Proof of Proposition \ref{propo1new}] We use Lemma \ref{propo1}. In fact, as both boundaries are accessible,  the probability of fixation of the efficient individuals with selective disadvantage is given by
\begin{eqnarray*}
\p_{y}(\{{\rm Fix.}\})&=& \p(Y_{T_{0,1}}=1 | Y_0 = y) \\
& = & \p(X_{T_{0,1}}=0 | X_0 = 1-y)=\frac{S(1)-S(1-y)}{S(1)-S(0)} \\
&=&\left\{ \begin{array}{ll}
C_{\kappa, \alpha}(1-(1-\kappa x)^{1-\frac{2\alpha}{\kappa}}) & \textrm{if}\quad 2\alpha\ne\kappa,\\
C_{\kappa, \alpha}\ln \left(\frac{1}{1-\kappa x}\right)& \textrm{if} \quad 2\alpha=\kappa,\\
\end{array}
\right.
\end{eqnarray*}
where $C_{\kappa, \alpha}=S(1)^{-1}$  (see for example Lemma 3.14 in \cite{eth}). 

\end{proof}

\section{The Ancestral Selection/Efficiency Graph}
\label{sectASEG}
\subsection{Properties of the ASEG}
\label{ASEG1}
This section is devoted to a more detailed study of the vertex counting process of the ASEG, $Z$.

When $\kappa <1$, it can be shown using standard techniques for birth-death processes, that when $\alpha>0$, the states $\{1, 2, 3, \ldots\}$ are positive recurrent  and the state $\{0\}$ is not accessible. When $\alpha=0$, the states $\{2, 3, \ldots\}$ are positive recurrent, $\{0\}$ is not accessible and  $\{1\}$ is  absorbing (see Proposition 1 in  \cite{GPP} for a detailed proof of the general case).  
In the ASEG, there are coalescence events, which correspond to negative jumps of size 1 of $Z$, and occur at the same rate as in the Kingman coalescent case.
 But there are also two types of events that create new vertices: branching events (which are related to selection, since their rate depends on $\alpha$) and pairwise branching events  (which  are related to efficiency, since their rate depends on $\kappa$). They correspond to positive jumps of size 1 of $Z$  and, therefore, 
$Z$ is not monotone as opposed to the block counting process of the Kingman coalescent which is  monotone decreasing. 
Having these  positive jumps of size one resembles the behavior of the vertex counting process of the ASG. 
% These positive jumps, together with the coloring rule of the ASEG, favor inefficient individuals over efficient individuals. 
When $\alpha=0$, by analogy with the ASG, the pairwise branching events, together with the coloring rule, seem to favor the creation of efficient individuals (as if there was some selection). However, the model is neutral and for every $t>0$,  $\mathbb{E}_{x}[X_t]=x$.
This apparent paradox has the following  nice interpretation. If $X$ starts at $x\in (0,1)$, it will eventually get absorbed in $1$ (with probability $x$) or  $0$ (with probability $1-x$). However, the process will go faster to 0 than to 1, since  the term $(1-\kappa x)$  slows the process down much more when $x$ is close to 1 than when its closer to zero. This explains why, for small times,  it is more likely to sample inefficient individuals but this apparent advantage vanishes for large times since  $X$ will have enough time to reach one of the two absorbing states.

\medskip 

The study of the case $\kappa = 1$ is particularly interesting, since it allows to understand the path behaviour of the Wright-Fisher model with efficiency under {\rm {\bf (M1)}} when $\kappa = 1$, which was not possible to study using the classical theory of diffusions. 
\begin{theorem}
If $\kappa=1$, the block counting process of the ASEG $Z$ is transient if and only if $\alpha\in (1/2, \infty) $. If $\alpha\in [0, 1/2)$ it is positive recurrent and has a unique stationary distribution $\mu$, which is a Sibuya distribution, i.e it is characterized as follows
$$
\sum_{j\ge 1}x^j\mu(j)=1-(1- x)^{1-2\alpha}.
$$
\label{thm2}
\end{theorem}
This theorem, together with the moment duality, allows us to describe the limiting behavior of $X$ and the distribution of the random variable $X_\infty:=\lim_{t\to\infty} X_t$ (in distribution),  when $\kappa=1$ and  $\alpha\in [0, 1/2)$.
\begin{corollary}\label{thmthm} The Wright-Fisher diffusion with efficiency  $X$ converges almost surely to zero if $\kappa=1$ and $\alpha> 1/2 $.  Moreover if $\kappa=1$ and  $\alpha\in [0, 1/2)$, then $X_\infty$ follows a Binomial distribution with parameter $1-(1- x)^{1-2\alpha}$.
\label{corothm2}
\end{corollary}

When $\kappa = 1$ and $\alpha = 1/2$, the path behavior of $X$ is not so easy to study. We conjecture that in this case the  block counting process of the ASEG is null recurrent and  $X$ goes to zero almost surely.

\subsection{Proofs}
This section is devoted to the proofs of the results presented in Sections \ref{introASEG} and \ref{ASEG1}.
\begin{proof}[Proof of Lemma \ref{Lemduality}]
Let us  denote by $\mathcal{Q}$  the generator of $Z$ which satisfies, for any 
$f$ bounded function of $\N$, that
\[
\mathcal{Q}f(n)  = \left(\alpha n +\frac{n(n-1)}{2} \kappa\right) \Big(f( n+1) - f(n)\Big) +  \frac{n(n-1)}{2}  \Big(f(n-1) - f(n)\Big).
\]
Recall that $\mathcal{A}$ denotes the infinitesimal generator of the Wright-Fisher diffusion with efficiency which satisfies \eqref{genX} for any function in $C^2([0,1])$.   We consider a function $h$ which is defined on $[0,1] \times \N$ and such that $h(x,n) = x^n$. 

Since  $h(x,n)$ and $\E[x^{Z_t}]$ are polynomials on $x$ and $h(x,n)$ and $\E[X_t^{n}]$ are bounded functions of $n$, we deduce from Proposition 1.2 of \cite{JK}  that our claim follows if we show that  for all $x \in [0,1]$ and $n \in \N$, the following identity holds 
$$\mathcal{A}h(x,n) = \mathcal{Q}h(x,n).$$
We observe that  in the left-hand side of the above identity,  $\mathcal{A}$ acts on $h$  (seen as a function of $x$) and in the right-hand side $\mathcal{Q}$ acts on $h$ (seen as a function of $n$). 

Hence from the definitions of $\mathcal{A}$ and $\mathcal{Q}$, it is clear that  for all $x \in [0,1]$ and $n \in \N$, we have
\[ 
\begin{split}
\mathcal{A}h(x,n) &= -\alpha x(1-x) \frac{\partial h(x,n)}{\partial x} + x(1-x)(1-\kappa x) \frac{\partial^2 h(x,n)}{\partial x^2} \\
& = - \alpha x(1-x) n x^{n-1} + x(1-x) (1-\kappa x) \frac{n(n-1)}{2} x^{n-2} \\
& =  (\alpha n +\frac{n(n-1)}{2} \kappa) (x^{n+1} - x^n) + \frac{n(n-1)}{2} (x^{n-1} - x^{n})\\
& = (\alpha n +\frac{n(n-1)}{2} \kappa) (h(x, n+1) - h(x, n)) +  \frac{n(n-1)}{2}  (h(x, n-1) - h(x, n)) \\ 
&= \mathcal{Q}h(x,n),
 \end{split}
\]
which  completes the proof.
\end{proof}

\begin{proof}[Proof of Theorem \ref{thm2}]
%In order to prove our result, we show the statement in Theorem 2'. That is to say that $Z$ is transient if and only if $\alpha>1/2$ and that $Z$ is positive recurrent when $\alpha\in[0,1/2)$ with invariant distribution $\mu$ satisfying
 %$$\sum_{j\ge 1}x^j\mu(j)=1-(1- x)^{1-2\alpha}, \qquad \textrm{for}\quad n\ge 1.$$
We start by proving that $Z$ is transient if and only if $\alpha>1/2$ and  $Z$ positive recurrent if $\alpha\in[0,1/2)$.
To do so, we study its jump chain  denoted by $S=(S_n, n\ge0)$ (see for instance  3.4.1 in  \cite{Norris}).  Observe that  $S$ is a birth-death Markov chain with transition probabilities given by 
\[
\Pi_{k,j}=\mathbb{P}(S_1=j|S_0=k)=\left\{ \begin{array}{ll}
\frac{1}{2}+\frac{\alpha}{2k}+O(k^{-2})& \textrm{ for }  j=k+1 ,\\
 \textrm{  }\\
\frac{1}{2}-\frac{\alpha}{2k}+O(k^{-2})& \textrm{ for }  j=k-1.
\end{array}
\right .
\]
From Theorem 3 in \cite{Ha} (see also \cite{Lamperti}) we know that $S$ is transient if and only if $\alpha>1/2$ which implies the first part of our claim.

For the second part, we use  the Foster-Lyapunov criteria (see for instance Proposition 1.3 in \cite{Hairer}) with the Lyapunov function $f(n)=\ln(n)$. Recall that $\mathcal{Q}$  is  the generator of $Z$. Using Taylor's expansion
\[
\begin{split}
\mathcal{Q}\ln(n)&=\alpha n \ln\left(1+\frac{1}{n}\right)+\binom{n}{2}\left(\ln\left(1-\frac{1}{n}\right)+\ln\left(1+\frac{1}{n}\right)\right)\\
&=\alpha\left(1-\frac{1}{2n}\right)-1/2\left(1+\frac{1}{n}\right)+O(n^{-2}).
\end{split}
\]
 That is to say
\[
\mathcal{Q}\ln(n)
\le -\epsilon,
\]
 for all but finitely many values of $n$, and  any $\epsilon\in (0,1/2-\alpha)$. According to Foster-Lyapunov criteria $Z$ is positive recurrent and there exists a unique invariant distribution here denoted by $\mu$. 
The moment duality property  (Lemma \ref{Lemduality}) implies that 
\begin{equation*}
\mathbb{E}[X_\infty^n |X_0 = x]=\sum_{j\ge 1}x^j\mu(j), \qquad \textrm{for all} \quad n\ge0.
\end{equation*}
Therefore $X_\infty$ must be a Bernoulli distribution and we have
 \[
\sum_{j\ge 1}x^j\mu(j)=\mathbb{P}_x(X_\infty=1)=\p_{1-x}(\{{\rm Fix.}\})=1-(1-\kappa x)^{1-\frac{2\alpha}{\kappa}}.
\]
The proof of our Theorem is now complete.
\end{proof}

\begin{proof}[Proof of Corollary \ref{corothm2}]
For the case $\alpha >1/2$, we observe  from the moment duality property  (Lemma \ref{Lemduality}) that $Z$ is transient implies that  $X$ goes to $0$ a.s. The case $\alpha<1/2$ has already been proved in the proof of  Theorem \ref{thm2}.
\end{proof}

\section{Scaling limit of model {\rm {\bf (M2)}}}
\label{S4}
Let us define the random variable $\underline{M}_n(N)=\inf\{i\in\mathbb{N}:C_{(n,i)}>N-1\}$, which counts the number of individuals created before efficient individuals are the only ones that can be produced. The following proposition  provides the limiting distribution of the amount of resource that  are still available after $\underline{M}_n(N)$ individuals have been produced.
\begin{proposition}\label{O}
Assume that   $\kappa=a/b\in(0,1)$ for some relative primes $a,b\in\mathbb{N}$. Let $U$ be a  uniform random variable defined on $\mathbb{D}_b:=\{j/b:j=0,1,...,b-1\}$. Then for all $n\ge 0$ and $x\in(0,1)$ , given $X^{(N)}_{n-1}=x$,
$$
\lim_{N\rightarrow\infty} N-C_{(n,\underline{M}_n(N))}\stackrel{d}{=}U,
$$
where ``$\stackrel{d}{=}$'' means identity in distribution or law.
\end{proposition} 
\begin{proof}
For $n\ge 0$ and $i\ge 0$, define the random variable $O^{(n)}_i:= \lceil C_{(n,i)} \rceil-C_{(n,i)}$, which measures the distance between the amount of resource consumed by the first $i$ individuals created in generation $n$ from its closest integer above. Assuming that the frequency of efficient individuals in generation $n-1$ is $x$, each new individual is efficient with probability $x$, in which case $O^{(n)}_{i+1}$ moves $a$ units from  $O^{(n)}_{i}$ on $\mathbb{D}_b$,  and inefficient otherwise, that is to say $O^{(n)}_{i+1}$ moves $b$ units from  $O^{(n)}_{i}$ on $\mathbb{D}_b$, and thus it does not move at all. In other words $(O^{(n)}_i, i\ge 0)$ is a Markov chain with state space $\mathbb{D}_b$ and transition probabilities given by
\begin{equation*}
P_{j/b,r/b}=\left\{ \begin{array}{ll}
x& \textrm{ if } r=j-a \textrm{ or }r=j-a+b ,\\
1-x, & \textrm{ if } r=j,\\
0, & \textrm{ otherwise. } 
\end{array}
\right .
\end{equation*}
where $j,r\in \{0,1,...,b-1\}$ .
Since the state space is finite, the Markov chain $(O^{(n)}_i, i\ge 0)$ has a  stationary distribution denoted by $U$ and since the transition probabilities from each state are the same, $U$ is the uniform distribution on $\mathbb{D}_b$. In other words, 
 $\lim_{i\rightarrow\infty} O_i\stackrel{d}{=}U$. Since $O_{\underline{M}_n(N)}=N-C_{(n,\underline{M}_n(N))}$ and $\underline{M}_n(N)\ge N$, we conclude that $\lim_{N\rightarrow\infty} O_{\underline{M}_n(N)} \stackrel{d}{=}U$ and the proof is complete.
\end{proof}
\begin{proof}[Proof of Theorem \ref{scaling4}]  We start by proving $(i)$, i.e. we consider the case where $\kappa$ is a rational number smaller than $1/2$. 
When it is no longer possible to produce inefficient individuals, there are two possible scenarios: either the amount of remaining resources is less than $1-\kappa$ and  it is not possible to produce more individuals of any type or  the amount of resource left is  in  $[1-\kappa,1),$ and  it is still possible to produce one more efficient individual.  By Proposition \ref{O}, the probability of the second case is asymptotically $\kappa$, as $N$ goes to infinity. Conditioning on the event that the amount of resource left is in $[1-\kappa,1)$, and given that the frequency of efficient individuals in the previous generation is $x$,  a new efficient individual  will be produced using the remaining resources with  probability $x$ i.e.  the number of new individuals produced follows a Bernoulli distribution with parameter $x$. Let us denote by $B$ such Bernoulli random variable with parameter $x$ and observe that, as $N$ increases, we have 
\[
\begin{split}
\mathbb{E}[X^{(N)}_1-x]&=\mathbb{E}\Big[X^{(N)}_1-x\Big| N-C_{(1,\underline{M}_1(N))}\in[1-\kappa,1)\Big]\\
&\hspace{5cm}\times\mathbb{P}\Big(N-C_{(1,\underline{M}_1(N))}\in[1-\kappa,1)\Big)\\
&+\mathbb{E}\Big[X^{(N)}_1-x\Big| N-C_{(1,\underline{M}_1(N))}\in[0,1-\kappa)\Big]\\
&\hspace{5cm}\times\mathbb{P}\Big(N-C_{(1,\underline{M}_1(N))}\in[0,1-\kappa)\Big)\\
&=\mathbb{E}\left[\frac{\sum_{i=1}^{\underline{M}_1(N)+B}(\mathbf{1}_{\{t(1,i)=0\}}-x)}{\underline{M}_1(N)+B}\right]\kappa\\
&\hspace{4cm}+\mathbb{E}\left[\frac{\sum_{i=1}^{\underline{M}_1(N)}(\mathbf{1}_{\{t(1,i)=0\}}-x)}{\underline{M}_1(N)}\right](1-\kappa).
\end{split}
\]
Note that any individual created after $\underline{M}_1(N)$ individuals  must be of the efficient type and
\[
\mathbb{E}\Big[\mathbf{1}_{\{t(1,i)=0\}}-x\Big]=0, \qquad\textrm{for every } \quad i\leq \underline{M}_1(N),
\] 
so the second term in the right-hand side is equal to zero and the first term becomes
\begin{eqnarray}
\mathbb{E}[X^{(N)}_1-x]
%&=\mathbb{E}\left[\frac{\sum_{i=1}^{\underline{M}_1(N)}(\mathbf{1}_{\{t(1,i)=0\}}-x)+B(1-x)}{\underline{M}_1(N)+B}\right]\kappa\frac{\mathbb{P}\Big(N-C_{(1,\underline{M}_1(N))}\in[1-\kappa,1)\Big)}{\kappa}\\
&=&\mathbb{E}\left[\frac{B(1-x)}{\underline{M}_1(N)+B}\right]\kappa \nonumber\\
&=&\frac{\kappa x (1-x)}{N_x}\mathbb{E}\left[\frac{B}{x}\frac{N_x}{\underline{M}_1(N)+B}\right]\nonumber\\
&\sim& \frac{\kappa x (1-x)}{N_x} = \frac{\kappa x (1-x)(1-\kappa x)}{N},
\label{equivY1}
\end{eqnarray}
where the equivalence comes from the fact that  $M_1 = \underline{M}_1(N)+B$ and  from Proposition \ref{LLN},  \[\lim_{N\rightarrow \infty}\E\left[\frac{B}{x} \frac{N_x}{\underline{M}_1(N)+B}\right]=1.\]
Next, as in the proof of Theorem \ref{thmscaling},  we observe that the drift and diffusive terms of the  SDE \eqref{diff4} are Lipschitz and H\"older continuous, respectively, which provides that  \eqref{diff4} has a unique strong solution. We denote  by $\mathcal{A}$ its infinitesimal generator, whose domain contains $C^2([0,1])$.
Similar arguments to those used   in the proof of Theorem \ref{thmscaling}, together with \eqref{equivY1}, lead to the fact that, if $f\in C^2([0,1])$ and $x \in [0,1]$,  the discrete generator of $X^{(N)}_{\lfloor Nt \rfloor}$ satisfies
\[
\begin{split}
A^Nf(x)&:=N\mathbb{E}\Big[f(X_1^{(N)})-f(x)\Big]\\
&=N\mathbb{E}\Big[X_1^{(N)}-x\Big]f'(x)+N\mathbb{E}\Big[(X_1^{(N)}-x)^2\Big]f''(x)+o(1)\\
&=\kappa x(1-x)(1-\kappa x)f'(x)+x(1-x)(1-\kappa x)f''(x)+o(1)\\
&\to \mathcal{A} f(x), 
\end{split}
\]
where, again, the term $o(1)$ depends on $x$ but converges to $0$ uniformly on $x$. 
Again, since all the processes involved are Feller taking values on $[0,1]$ and the convergence of the generators is uniform,   the result follows from Lemma 17.25 of \cite{Kallemberg}.

\medskip 

We now prove $(ii)$, i.e. we consider the case where $1-\kappa = 1/b$ for some $b \in \N$. 
 Let $\{G_i, 1\le i\le b-1\}$  be a sequence of independent random variables, such that $\mathbb{P}(G_i=j)=x^j(1-x)$ for all $j\in \{0,1,...,i-1\}$ and $\mathbb{P}(G_i=i)=x^i$. In other words,  $G_i$ is a geometric random variable truncated at $i$ which  is interpreted as the number of efficient individuals produced when the amount of  resource left is $i/b$. We have $\mathbb{E}[G_r]=x(1-x^r)(1-x)^{-1}$. Using Proposition \ref{O} and similar arguments as in the proof of \eqref{equivY1} allow us to deduce
 
\[
\begin{split}
\mathbb{E}\Big[X^{(N)}_1 -x\Big]&=\sum_{r=0}^{b-1}\mathbb{E}\Big[X^{(N)}_1-x\Big| N-C_{(1,\underline{M}_1(N))}\in [\frac{r}{b},\frac{r+1}{b})\Big]\\
&\hspace{3cm}\times\mathbb{P}\Big(N-C_{(1,\underline{M}_1(N))}\in  [\frac{r}{b},\frac{r+1}{b})\Big)\\
&=\sum_{r=1}^{b-1} \mathbb{E}\left[\frac{\sum_{i=1}^{\underline{M}_1(N)}(\mathbf{1}_{\{t(1,i)=0\}}-x)+G_r(1-x)}{\underline{M}_1(N)+G_r}\right](1-\kappa)\\
&=\frac{(1-\kappa) (1-x)}{N_x}\sum_{r=1}^{b-1} \mathbb{E}\left[G_r\frac{N_x}{\underline{M}_1(N)+G_r}\right]\\
&\sim\frac{(1-\kappa)(1-\kappa x) (1-x)}{N}\sum_{r=1}^{b-1}x\frac{1-x^r}{1-x}\\
&=\frac{(1-\kappa)(1-\kappa x) x}{N}\sum_{r=1}^{b-1}(1-x^r).
\end{split}
\]
To prove our result we proceed similarly as in part $(i)$. The SDE \eqref{diff5} has a unique strong solution and we denote its infinitesimal generator by $\mathcal{A}$. If $f\in C^2([0,1])$ and $x \in [0,1]$, the discrete generator of $X^{(N)}_{\lfloor Nt \rfloor}$ satisfies
\[
\begin{split}
A^Nf(x)&=(1-\kappa) x(1-x)(1-\kappa x)\sum_{j=1}^{b-1}(1-x^j)f'(x)\\
&\hspace{5cm}+x(1-x)(1-\kappa x)f''(x)+o(1)\\
&\to \mathcal{A} f(x),
\end{split}
\]
and the result follows, as in part $(ii)$.

\medskip 

Finally, we prove part $(iii)$, i.e. the case where $1-\kappa=a/b\in(0,1/2)$ for some relative primes $a,b\in\mathbb{N}$. We recall that  $m=\lfloor(1-\kappa)^{-1}\rfloor$, $c_m=1-ma/b$ and $c_i=a/b$ for all $i=1,2,...,m-1$. Let  $\{G_i, 1\le i \le m\}$  be a sequence of independent random variables defined as in part $(ii)$. The constants $\{c_i, 1\le i\le m\}$  and the random variables $\{G_i, 1\le i \le m\}$ have the following interpretation:  once it is not longer  possible to produce more inefficient individuals, if the amount of remaining resource lies in $[i a/b, (i+1)a/b)$  the number of new individuals produced will be $G_i$ and the probability of such event is asymptotically $c_i$. Similarly,  if the amount of remaining resource is in $[ma/b, 1]$  the number of new individuals produced will be  $G_m$ and the probability of such event is asymptotically $c_m$ (see Figure \ref{ab}).

\begin{figure}[h]
\begin{center}
\includegraphics[width = 1.0\textwidth]{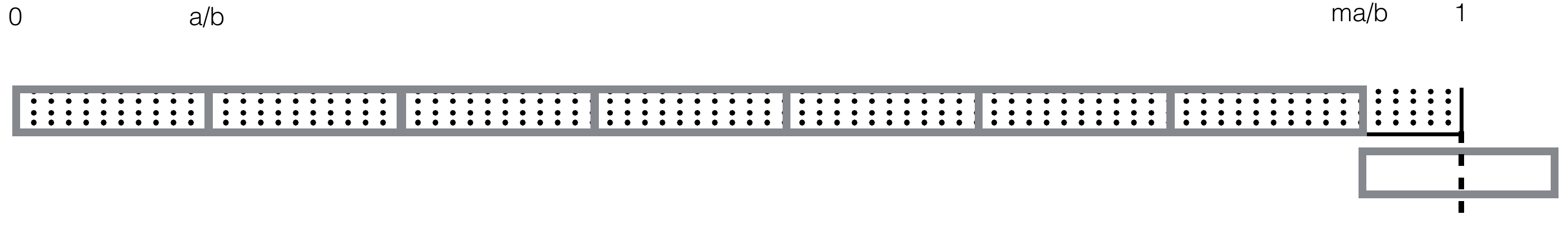}
\caption{\small  
In this example $1-\kappa=a/b=11/83$, the number of full rectangles that can be fit in the line of length 1 is $m=7$ and  $c_m=1-m(1-\kappa)=(83-77)/83=6/83.$}
\end{center}
\label{ab}
\end{figure}
Proceeding as in parts $(i)$ and $(ii)$ we have
\[
\begin{split}
\mathbb{E}\Big[X^{(N)}_1-x\Big]&=\sum_{r=0}^{m}\mathbb{E}\Big[X^{(N)}_1-x\Big| N-C_{(1,\underline{M}_1(N))}\in [ra/b,(r+1)a/b)\Big]\\
&\hspace{3.5cm}\times\mathbb{P}\Big(N-C_{(1,\underline{M}_1(N))}\in [ra/b,(r+1)a/b)\Big)\\
&=\sum_{r=1}^{m} \mathbb{E}\left[\frac{\sum_{i=1}^{\underline{M}_1(N)}(\mathbf{1}_{\{t(1,i)=0\}}-x)+G_r(1-x)}{\underline{M}_1(N)+G_r}\right]c_r\\
&=\frac{1-x}{N_x}\sum_{r=1}^{m}c_r  \mathbb{E}\left[G_r\frac{N_x}{\underline{M}_1(N)+G_r}\right]\\
&\sim \frac{(1-\kappa x) (1-x)}{N}\sum_{r=1}^{m}c_rx\frac{1-x^r}{1-x}\\
&=\frac{(1-\kappa x) x}{N}\sum_{r=1}^{m}c_r(1-x^r).
\end{split}
\]
To complete the proof,  we proceed similarly as in parts $(i)$ and $(ii)$. Again, the SDE \eqref{diff56}  has a unique strong solution. We denote by $\mathcal{A}$ its infinitesimal generator, whose domain contains $C^2([0,1])$. If $f\in C^2([0,1])$ and $x \in [0,1]$, the discrete generator of $X^{(N)}$ satisfies
\[
\begin{split}
A^Nf(x)&= x(1-x)(1-\kappa x)\sum_{j=1}^{b-1}c_j(1-x^j)f'(x)\\
&\hspace{4cm}+x(1-x)(1-\kappa x)f''(x)+o(1)\\
&\to  \mathcal{A} f(x),
\end{split}
\]
and again, the conclusion follows.
\end{proof}

Before proving Proposition \ref{propoM2}, we start by proving two other results on the path behavior of the Wright-Fisher diffusion with efficiency under  {\rm {\bf (M2)}}.
\begin{lemma}
Let $X=(X_t, t\ge 0)$ be the unique strong solution of \eqref{diff4} parametrised by a rational number $\kappa \in (0,1)$ and $\alpha \ge 0$.
The boundary points $0$ and $1$ are accesible.
\end{lemma}

\begin{proof}
Let $X=(X_t, t\ge 0)$ be the unique strong solution of \eqref{diff4} parametrised by $\alpha \ge 0$  and $\kappa$ a rational number in $(0,1)$.  For $c \in \R$, consider the process $\overline Y^{(c)} = (\overline Y^{{(c)}}_t, t\geq0)$ defined in the proof of Lemma \ref{propo1}. Recall that both boundaries 0 and 1 are accessible for this process.  Again, we use a stochastic domination argument. If $\alpha >0$, we have, a.s.
$$ \overline Y^{(\kappa)}_t \ge X_t \ge \overline Y ^{(\alpha/(\kappa -1))}_t, \ t\ge 0,$$
and if $\alpha = 0, X \equiv \overline Y^{(\kappa)}$, and the conclusion follows.
\end{proof}
 
\begin{proposition}\label{TM2} Let $X=(X_t, t\ge 0)$ be the unique strong solution of \eqref{diff4} parametrised by $\alpha \ge 0$  and $\kappa$ a rational number in $(0,1)$. Then $$ \mathbb{E}[T_{0,1} | X_0 = x]<\infty.$$
\end{proposition} 

\begin{proof}
We follow closely the proof of Proposition \ref{T}. Recall that the Wright-Fisher diffusion with efficiency under {\rm {\bf (M1)}}  or  {\rm {\bf (M2)}}  have the same infinitesimal variance, so the Green's function associated to the Wright-Fisher diffusion with efficiency under  {\rm {\bf (M2)}}  also satisfies equation \eqref{green} (where $S$ is its scale function). The conclusion follows by the same arguments that are used to prove item $(ii)$ of Proposition \ref{T}.
\end{proof}

\begin{proof}[Proof of Proposition \ref{propoM2}]
By the same arguments as in the proof of Proposition \ref{propo1new}, we have
$$\p_{y}(\{{\rm Fix.}\})=\p(X_{T_{0,1}}=0 | X_0 = 1-y)=\frac{S(1)-S(1-y)}{S(1)-S(0)}, $$
 where $S$ is the scale function of the Wright-Fisher diffusion with efficiency  {\rm {\bf (M2)}} , which, for $x \in [0,1]$ is given by 
 \begin{eqnarray*}
S(x)&=&\int_0^x \exp\left\{-2\int_\theta^u \left(\kappa - \frac{\alpha }{1- \kappa v}\right) \ud v \right\}\ud u \\
& = & K \int_0^x \exp(-2 \kappa u) (1-\kappa u)^{-2 \alpha/\kappa})\ud u,
\end{eqnarray*}
 where $\theta$ is an arbitrary positive number and $K$ is a constant that depends on $( \kappa, \alpha, \theta)$. 
\end{proof}

 \noindent \textbf{Acknowledgements.} 
 All authors would like to thank Ximena Escalera, Jos\'e Carlos Ram\'on Hern\'andez and Fernanda L\'opez for carefully reading a preliminary version of this paper and for many useful discussions. We want to thank as well the two anonymous referees whose careful reading led to significant improvements. 

 JCP  acknowledges support from  the Royal Society and CONACyT (CB-250590). This work was concluded whilst JCP was on sabbatical leave holding a David Parkin Visiting Professorship  at the University of Bath, he gratefully acknowledges the kind hospitality of the Department and University. AGC acknowledges support from UNAM (PAPIIT IA100419) and CONACYT (CB-A1-S-14615). VMP acknowledges support from the DGAPA-UNAM postdoctoral program.

\end{document}